\documentclass[reqno,a4paper]{amsart}

\usepackage[latin1]{inputenc}
\usepackage[english]{babel}
\usepackage{eucal,amsfonts,amssymb,amsmath,epsfig,mathrsfs}
\usepackage{cancel,soul}
\usepackage{color}
\usepackage{amscd,amsxtra}
\usepackage{enumerate}
\usepackage{latexsym}
\usepackage{bm}

\allowdisplaybreaks

\newcounter{ipotesi}

\Alph{ipotesi}
 \makeatletter \@addtoreset{equation}{section}

\makeatother \makeatletter

\newtheorem{thm}{Theorem}
\newtheorem{hyp}[thm]{Hypotheses}{\rm}
{\rm}
\newtheorem{lemm}[thm]{Lemma}
\newtheorem{cor}[thm]{Corollary}

\newtheorem{prop}[thm]{Proposition}
\newtheorem{defi}[thm]{Definition}
\newtheorem{rmk}[thm]{Remark}{\rm}

\newcounter{parentenv}

\newcommand{\R}{{\mathbb R}}

\newcommand{\E}{{\mathbb E}}
\newcommand{\N}{{\mathbb N}}

\newcommand{\X}{{\mathcal{X}}}
\newcommand{\D}{{D}}
\newcommand{\J}{{\mathcal{D}}}

\newcommand{\eps}{\varepsilon}
\newcommand{\ra}{\rightarrow}

\newcommand{\Id}{{\rm {Id}}}

\newcommand{\lip}{\operatorname{{Lip}}}

\renewcommand{\hat}[1]{\widehat{#1}}

\newcommand{\Tr}{{\operatorname{Tr}}}

\newcommand{\set}[1]{{\left\{#1\right\}}}
\newcommand{\pa}[1]{{\left(#1\right)}}
\newcommand{\sq}[1]{{\left[#1\right]}}
\newcommand{\gen}[1]{{\left\langle #1\right\rangle}}
\newcommand{\abs}[1]{{\left|#1\right|}}
\newcommand{\norm}[1]{{\left\|#1\right\|}}
\newcommand{\scal}[2]{{\left\langle #1,#2\right\rangle}}

\newcommand{\eqsys}[1]{{\left\{\begin{array}{ll}#1\end{array}\right.}}

\begin{document}

\frenchspacing

\title[Regularizing properties of transition semigroups]{Regularizing properties of (non-Gaussian) transition semigroups in Hilbert spaces}

\author[D. A. Bignamini]{{D. A. Bignamini}}

\author[S. Ferrari]{{S. Ferrari}$^*$}\thanks{$^*$Corresponding Author}

\address[D. A. Bignamini]{Dipartimento di Scienze Matematiche, Fisiche e Informatiche, Universit\`a degli Studi di Parma, Parco Area delle Scienze 53/A, 43124 Parma, Italy.}
\email{\textcolor[rgb]{0.00,0.00,0.84}{davideaugusto.bignamini@unimore.it}}

\address[S. Ferrari]{Dipartimento di Matematica e Fisica ``Ennio De Giorgi'', Universit\`a del Salento. POB 193, 73100 Lecce,
Italy.}
\email{\textcolor[rgb]{0.00,0.00,0.84}{simone.ferrari@unisalento.it}}



\date{\today}

\begin{abstract}
Let $\X$ be a separable Hilbert space with norm $\norm{\cdot}$ and let $T>0$. Let $Q$ be a linear, self-adjoint, positive, trace class operator on $\X$, let $F:\X\ra\X$ be a (smooth enough) function and let $W(t)$ be a $\X$-valued cylindrical Wiener process. For $\alpha\in [0,1/2]$ we consider the operator $A:=-(1/2)Q^{2\alpha-1}:Q^{1-2\alpha}(\X)\subseteq\X\ra\X$. We are interested in the mild solution $X(t,x)$ of the semilinear stochastic partial differential equation
\begin{gather*}
\eqsys{
dX(t,x)=\big(AX(t,x)+F(X(t,x))\big)dt+ Q^{\alpha}dW(t), & t\in(0,T];\\
X(0,x)=x\in \X,
}
\end{gather*}
and in its associated transition semigroup 
\begin{align*}
P(t)\varphi(x):=\E[\varphi(X(t,x))], \qquad \varphi\in B_b(\X),\ t\in[0,T],\ x\in \X;
\end{align*}
where $B_b(\X)$ is the space of the bounded and Borel measurable functions. We will show that under suitable hypotheses on $Q$ and $F$, $P(t)$ enjoys regularizing properties, along a continuously embedded subspace of $\X$. More precisely there exists $K:=K(F,T)>0$ such that for every $\varphi\in B_b(\X)$, $x\in \X$, $t\in(0,T]$ and $h\in Q^\alpha(\X)$ it holds
\[|P(t)\varphi(x+h)-P(t)\varphi(x)|\leq Kt^{-1/2}\|Q^{-\alpha}h\|.\] 
\end{abstract}

\maketitle

\section{Introduction}\label{introduction}

Let $(\Omega,\mathcal{F},\set{\mathcal{F}_t}_{t\geq 0},\mathbb{P})$ be a complete filtered probability space. We denote by $\E[\cdot]$ the expectation with respect to $\mathbb{P}$. Let $\X$ be a real separable Hilbert space with inner product $\scal{\cdot}{\cdot}$ and norm $\norm{\cdot}$. Let $Q$ be a linear, self-adjoint, positive, trace class operator on $\X$. For $\alpha\in [0,1/2]$ we consider the operator $A:=-(1/2)Q^{2\alpha-1}: Q^{1-2\alpha}(\X)\subseteq \X\rightarrow \X$, and a suitable (smooth enough) function $F:\X\ra\X$. Let $W(t)$ be a $\X$-valued cylindrical Wiener process (see Remark \ref{MBX}). 

For $T>0$ we consider the mild solution $X(t,x)$ of the semilinear stochastic partial differential equation 
\begin{gather}\label{eqF}
\eqsys{
dX(t,x)=\big(AX(t,x)+F(X(t,x))\big)dt+ Q^{\alpha}dW(t), & t\in(0,T];\\
X(0,x)=x\in \X,
}
\end{gather}
and its associated transition semigroup 
\begin{align}\label{transition}
P(t)\varphi(x):=\E[\varphi(X(t,x))]\qquad t\in[0,T],\ x\in \X;
\end{align}
where $\varphi\in B_b(\X)$ (the space of the real-valued, bounded and Borel measurable functions). By mild solution of (\ref{eqF}) we mean that for every $x\in\X$ there exists a $\X$-valued adapted stochastic process $\set{X(t,x)}_{t\geq 0}$ satisfying the mild form of \eqref{eqF}, namely for $x\in\X$ and $t\in[0,T]$ it holds
\begin{equation}\label{mildF}
X(t,x)=e^{tA}x+\int_0^te^{(t-s)A}F(X(s,x))ds+\int_0^te^{(t-s)A}Q^{\alpha}dW(s),
\end{equation}
and such that $\mathbb{P}(\int_0^T\norm{X(s,x)}^2ds<+\infty)=1$, for any $x\in\X$. 
The aim of this paper is to show that, under suitable assumptions, the semigroup $P(t)$, defined in \eqref{transition}, maps $B_b(\X)$ into the space of Lipschitz continuous functions along an appropriate continuously embedded subspace of $\X$. To be more precise we introduce some notations and hypotheses. We say that a function $\varphi:\X\ra \R$ is $Y$-Lipschitz, where $Y$ is a continuously embedded subspace of $\X$ with norm $\norm{\cdot}_Y$, if there exists $L>0$ such that for every $x\in\X$ and $y\in Y$
\[
\abs{\varphi(x+y)-\varphi(x)}\leq L\norm{y}_Y.
\]
\begin{hyp}\label{hyp0}
Let $T>0$ and let $\alpha\in [0,1/2]$. Let $\X$ be a real separable Hilbert space with inner product $\scal{\cdot}{\cdot}$ and norm $\norm{\cdot}$. We assume that
\begin{enumerate}[\rm(i)]
\item \label{hyp0.1} $Q$ is a linear, self-adjoint, (strictly) positive, trace class operator on $\X$ and let
\[A:=-(1/2)Q^{2\alpha-1}:Q^{1-2\alpha}(\X)\subseteq \X\ra\X;\]


\item\label{hyp0.3}  there exists $\gamma\in(0,1)$ such that for any $t\in(0,T]$
\begin{align}\label{PD}
\int_0^ts^{-\gamma}\Tr[e^{2sA}Q^{2\alpha}] ds<+\infty,
\end{align}
where $\Tr$ denotes the trace operator (see \eqref{traccie}).
\end{enumerate}
\end{hyp}

\noindent We remark that by Hypothesis \ref{hyp0}\eqref{hyp0.1}, \cite[Section II Corollary 4.7]{EN-NA2} and \cite[Theorem 2.3.15]{EN-NA1}, $A$ is the infinitesimal generator of a strongly continuous, analytic and contraction semigroup $e^{tA}$ on $\X$. Hypothesis \ref{hyp0}\eqref{hyp0.3} is standard in the literature, since it guarantees that the mild solution of \eqref{eqF} is path-continuous. We remark that this condition may appear different from the one in \cite[Theorem 5.11]{DA-ZA4}, because the authors of \cite{DA-ZA4} use a Hilbert--Schmidt norm for operators from $Q^{1/2}(\X)$ to $\X$. Their condition becomes \eqref{PD} in our case. We stress that since $Q^\alpha$ can be written as a negative power of the operator $-A$, then \eqref{eqF} can be rewritten in the following way 
\begin{gather*}
\eqsys{
dX(t,x)=\big(AX(t,x)+F(X(t,x))\big)dt+ (-2A)^{\alpha/(2\alpha-1)}dW(t), & t\in(0,T];\\
X(0,x)=x\in \X,
}
\end{gather*}
where the smoothing effect of the diffusion on the noise is more evident.

Hypotheses \ref{hyp0} and the assumption of Lipschitz continuity of $F$ are classical for the study of the existence and uniqueness of the solution of \eqref{eqF}. Moreover, for $\alpha\in [0,1/2)$, if some further conditions are assumed, it is possible to prove that for any $t\in(0,T]$
\begin{equation}\label{strong1}
P(t)\left(B_b(\X)\right)\subseteq \lip_b(\X)
\end{equation}
where $\lip_b(\X)$ is the space of the bounded and Lipschitz continuous functions on $\X$. 
There is a vast literature dealing with similar types of smoothing properties. See for example \cite{LOR1,MET-PAL-WAC1,TAI1} for an overview in the finite dimensional case and \cite{BI-FE1,DA-ZA1,PES-ZA1,ZAB1} for the infinite dimensional case.

The main result of this paper is a regularization result similar to \eqref{strong1} for the transition semigroup $P(t)$, defined in \eqref{transition}, with some non-standard hypotheses on $F$.
Let $H_{\alpha}:=Q^{\alpha}(\X)$ and for every $h,k\in H_{\alpha}$
\begin{align*}
\scal{h}{k}_\alpha:=\langle Q^{-\alpha}h,Q^{-\alpha}k\rangle,
\end{align*}
\noindent then $(H_{\alpha},\scal{\cdot}{\cdot}_\alpha)$ is a Hilbert space continuously embedded in $\X$. We denote by $\norm{\cdot}_{\alpha}$ the norm induced by $\scal{\cdot}{\cdot}_\alpha$ on $H_\alpha$. 
Our setting is similar to the one of \cite{ES-ST1}, although there a different problem (existence of an invariant measure for a stochastic Cahn--Hilliard type equation) was considered.

\begin{defi}
Let $Y,Z$ be two Hilbert spaces, endowed with the norms $\norm{\cdot}_Y$ and $\norm{\cdot}_Z$ respectively, and let $\Phi:\X\ra Z$ be a Borel measurable function. Assume that $Y$ is continuously embedded in $\X$. We say that $\Phi$ is $Y$-Lipschitz when there exists $C>0$ such that for every $x\in\X$ and $y\in Y$
\begin{equation}\label{LipY}
\norm{\Phi(x+y)-\Phi(x)}_Z\leq C\norm{y}_Y.
\end{equation}
We denote by $\lip_{Y}(\X;Z)$ the sets of Borel measurable, $Z$-valued and $Y$-Lipschitz functions, and by $\lip_{b,Y}(\X;Z)$ the subset of $\lip_{Y}(\X;Z)$ consisting of bounded functions. If $Z=\R$ we simply write $\lip_{b,Y}(\X)$. We call $Y$-Lipschitz constant of $\Phi$ the infimum of all the constants $C>0$ verifying \eqref{LipY}.
\end{defi}
Now we state the hypotheses we will use throughtout the paper.
 
\begin{hyp}\label{hyp1}
Let Hypotheses \ref{hyp0} hold true and let $F:\X\rightarrow\X$ be a Borel measurable (possibly unbounded) function such that
\begin{enumerate}[\rm(i)]
\item $F(\X)\subseteq H_{\alpha}$ and $F$ is $H_{\alpha}$-Lipschitz, with $H_\alpha$-Lipschitz constant $L_{F,\alpha}$;
\item \label{opl} if $\alpha\in [1/4,1/2]$, then we assume that $F:\X\ra H_\alpha$ is locally bounded.
\end{enumerate} 
\end{hyp}
Let us make some considerations about these assumptions. The requirement that $F(\X)$ is contained in $H_\alpha$ is not uncommon, for example the case $F=-Q^{2\alpha}\D U$ where $U:\X\rightarrow \R$ is a suitable convex function, often appears in the literature (see \cite{AD-CA-FE1,AN-FE-PA1,CAP-FER1,CAP-FER2,FER1} for $\alpha=1/2$, \cite{DA2,DA-LU2,DA-LU3} for $\alpha=0$ and \cite{DA-TU1} for general $\alpha$). A condition similar to Hypothesis \ref{hyp1}\eqref{opl} was already considered in \cite{DA-FL2}. 
\noindent We stress that Hypothesis \ref{hyp1} does not imply that $F$ is continuous. 
In Section \ref{exA} we show some examples of functions $F$ satisfying Hypotheses \ref{hyp1}. Now we state the main result of this paper.

\begin{thm}\label{Hstrong}
Assume that Hypotheses \ref{hyp1} hold. Then, for any $t\in(0,T]$, the semigroup $P(t)$ maps the space $B_b(\X)$ to $\lip_{b,H_\alpha}(\X)$. More precisely for every $\varphi\in B_b(\X)$, $x\in\X$, $h\in H_{\alpha}$ and $t\in (0,T]$ it holds
\[\vert P(t)\varphi(x+h)-P(t)\varphi(x)\vert\leq \frac{e^{L_{F,\alpha}T}}{\sqrt{t}}\Vert\varphi\Vert_{\infty}\norm{h}_{\alpha}.\]
\end{thm}
For $\alpha\in[0,1/2)$, the regularization result of Theorem \ref{Hstrong} is weaker than \eqref{strong1}, but we emphasize that we do not assume that $F$ is Lipschitz continuous on $\X$.
We were also interested to see what result can be obtained by assuming more standard assumptions on $F$. In Section \ref{proofH} we are going to use the same techniques as in the proof of Theorem \ref{Hstrong} to the case in which $F$ is Lipschiz continuous, and we prove the following result.

\begin{thm}\label{Hstrongvar}
Assume that Hypotheses \ref{hyp0} hold. Let $F:\X\rightarrow\X$ be a function such that $F(\X)\subseteq H_{\alpha}$ and $Q^{-\alpha}F$ is Lipschitz continuous with Lipschitz constant $K_{F,\alpha}$. Then
\begin{enumerate}[\rm (a)]
\item\label{furmi} for any $t\in(0,T]$, $P(t)(B_b(\X))\subseteq \lip_b(\X)$, if $\alpha\in [0,1/2)$;
\item for any $t\in(0,T]$, $P(t)(B_b(\X))\subseteq \lip_{b,H_{1/2}}(\X)$, if $\alpha=1/2$.
\end{enumerate}
\end{thm}
\noindent Statement \eqref{furmi} of Theorem \ref{Hstrongvar} was already proved in \cite{BON-FUR1,FUR1}. However, for $\alpha\in [0,1/4)$, our proof is simpler, because we can exploit the identity $A=-(1/2)Q^{2\alpha-1}$ and the analyticity of the semigroup $e^{tA}$. Instead the case $\alpha=1/2$ is not covered by \cite{BON-FUR1,FUR1}. In the papers \cite{FE-GO2,FE-GO1,MAS2,MAS1} the case $\alpha=1/2$ is considered, but, as we shall see in Section \ref{FUR}, the authors consider a different concept of derivative compared to the one presented in this paper in Section \ref{SpacialRegularity}.

Before proceeding we want to make some considerations about the results of this paper and some of the results already appeared in the literature. In \cite[Section 7.7]{DA-ZA1}, \cite{ES1} and \cite{PES-ZA1} the authors study a more general stochastic partial differential equation than \eqref{eqF}, but in our case their assumptions imply that $Q^{\alpha}$ has a continuous inverse, and in infinite dimension it makes sense only when $\alpha=0$, since $Q^{-\theta}$ is unbounded for every $\theta>0$. In \cite{CER1} the author proves \eqref{strong1} in an important case, namely when $A$ is the realization of a second order differential operator in $L^2(\Omega,d\xi)$ ($\Omega$ is an appropriate domain of $\R^n$, for some $n\in\N$, and $d\xi$ is the Lebesgue measure), and $F$ satisfies some technical conditions. In \cite{BON-FUR1,FUR1,MAS-SEI1} the authors work in a more general setting. However, the case $\alpha=1/2$ is not covered by their theory. Indeed one of the fundamental hypotheses assumed in \cite{BON-FUR1,FUR1,MAS-SEI1} is the following: for any $t\in(0,T]$
\begin{align}
e^{tA}(\X) \subseteq Q_t^{1/2}(\X),\label{Fii}
\end{align}
where $Q_tx=\int^t_0 e^{2sA}Q^{2\alpha}xds$. If Hypotheses \ref{hyp0} hold true and $\alpha\in [0,1/2)$, then \eqref{Fii} is verified. Indeed, in our case, $Q_t=Q(\Id-e^{2tA})$ and recalling that by the analyticity of $e^{tA}$ it holds that for any $t\in(0,T]$, the range of $e^{tA}$ is contained in the domain of $A^k$ for every $k\in\N$ (see \cite[Proposition 2.1.1(i)]{LUN1}), it is sufficient to prove that $(\Id-e^{2tA})$ is invertible. Since $2A$ is negative, we have $\|e^{2tA}\|_{\mathcal{L}(\X)}<1$, and so $(\Id-e^{2tA})$ is invertible. In particular $Q_t^{1/2}(\X)= Q^{1/2}(\X)$ and so we get \eqref{Fii}. Instead for $\alpha=1/2$ condition \eqref{Fii} is not verified, because $A=-(1/2)\Id$ and so $e^{-(1/2)t\Id}\X=\X$, for any $t\in(0,T]$. Moreover in \cite{BON-FUR1,FUR1,MAS-SEI1} the authors assume that $F$ is Lipschitz continuous, while our Hypotheses \ref{hyp1} do not imply Lipschitz continuity of $F$. 


The paper is organized in the following way: in Section \ref{lip} we introduce the notation we will use throughout the paper and recall some classical results for the stochastic partial differential equation \eqref{eqF} when $F$ is a Lipschitz continuous function. In Section \ref{ExAndUni} we show that when Hypotheses \ref{hyp1} hold true, then \eqref{eqF} admits a unique mild solution. In Section \ref{SpacialRegularity} we introduce a gradient operator along $H_\alpha$ and show that the mild solution of \eqref{eqF}, when Hypotheses \ref{hyp1} hold true, enjoys some regularity properties with respect to this gradient operator. In Section \ref{StrongFellerProperty} we show a new version of the Bismuth--Elworthy--Li formula and we prove Theorem \ref{Hstrong}. Section \ref{proofH} is dedicated to the proof of Theorem \ref{Hstrongvar}. In Section \ref{FUR} we will compare our results to those in the literature. Finally in Section \ref{exA} we will show some examples of functions $F$ satisfying Hypotheses \ref{hyp1} or the hypotheses of Theorem \ref{Hstrongvar}. In particular, we shall consider abstract Cahn--Hilliard type equations such as in \cite{ES-ST1}.

\section{Notation and preliminary results}\label{lip}

Let $H_1$ and $H_2$ be two real Hilbert spaces with inner products $\gen{\cdot,\cdot}_{H_1}$ and $\gen{\cdot,\cdot}_{H_2}$ respectively. We denote by $\mathcal{B}(H_1)$ the family of the Borel subsets of $H_1$ and by $B_b(H_1;H_2)$ the set of the $H_2$-valued, bounded and Borel measurable functions. We denote by $C^k_b(H_1;H_2)$, $k\geq 0$ the set of the $k$-times Fr\'echet differentiable functions from $H_1$ to $H_2$ with bounded derivatives up to order $k$. If $H_2=\R$ we simply write $C_b^k(H_1).$ For a function $\Phi\in C_b^1(H_1;H_2)$ we denote by $\J \Phi(x)$ the derivative operator of $\Phi$ at the point $x\in H_1$. If $f\in C_b^1(H_1)$, for every $x\in H_1$ there exists a unique $k\in H_1$ such that for every $h\in H_1$
\[\J f(x)(h)=\gen{h,k}_{H_1}.\]
We let $\D f(x):=k$. If $\Phi:H_1\ra H_2$ is Gateaux differentiable we denote by $\J^G\Phi(x)$ the Gateaux derivative operator of $\Phi$ at the point $x\in H_1$. See \cite[Chapter 7]{FAB1}.

Let $B\in\mathcal{L}(\X)$ (the set of bounded linear operators from $\X$ to itself). We say that $B$ is \emph{non-negative} (\emph{positive}) if for every $x\in \X\setminus\set{0}$
\[\gen{Bx,x}\geq 0\ (>0).\]
In the a same way we define the non-positive (negative) operators. We recall that a non-negative and self-adjoint operator $B\in\mathcal{L}(\X)$ is a \emph{trace class operator} whenever 
\begin{align}\label{traccie}
\Tr[B]:=\sum_{n=1}^{+\infty}\scal{Be_n}{e_n}<+\infty,
\end{align}
for some (and hence, every) orthonormal basis $\{e_n\}_{n\in\N}$ of $\X$. We recall that the trace is independent of the choice of the basis. See \cite[Section XI.6 and XI.9]{DUN-SCH2}.


Let $(\Omega,\mathcal{F},\{\mathcal{F}_t\}_{t\geq 0},\mathbb{P})$ be a complete, filtered probability space. We denote by $\mathbb{E}[\cdot]$ the expectation with respect to $\mathbb{P}$. Let $Y$ be a Banach space. If $\xi:(\Omega,\mathcal{F},\mathbb{P})\ra (Y,\mathcal{B}(Y))$ is a random variable, we denote by 
\[\mathscr{L}(\xi):=\mathbb{P}\circ\xi^{-1}\] 
the law of $\xi$ on $(Y,\mathcal{B}(Y))$. Throughout the paper when we refer to a process we mean a process defined on $(\Omega,\mathcal{F},\{\mathcal{F}_t\}_{t\geq 0},\mathbb{P})$. Now we make a remark about the notion of Wiener process in a Hilbert space.

\begin{rmk}\label{MBX}
Let $E$ be a separable Hilbert space and let $S$ be a self-adjoint and positive operator from $E$ to itself. If $S$ is a trace class operator, we call genuine $E$-valued Wiener process with $S$ as covariance operator a $E$-valued adapted process $\{W(t)\}_{t\geq 0}$ such that
\begin{enumerate}[\rm(i)]
\item $W(0)=0$ and $\mathscr{L}(W(t)-W(s))$ is the Gaussian measure with mean zero and covariance operator $(t-s)S$ on $E$, where $S$ is an operator satisfying Hypothesis \ref{hyp0}\eqref{hyp0.1}.\label{MBX1}

\item for $n\in\N$ and $0\leq t_1<t_2<\cdots<t_n$ the random variables $W(t_1)$, $W(t_2)-W(t_1)$,\ldots, $W(t_n)-W(t_{n-1})$ are independent;\label{MBX2}

\item for $\mathbb{P}$-almost every $\omega\in\Omega$, $W(\cdot,\omega)$ is a continuous function on $[0,+\infty)$. This condition is called path-continuity.\label{MBX3}
\end{enumerate}
If $S$ is not a trace class operator, it is however possible to define a generalized $E$-valued Wiener process with $S$ as covariance matrix (see \cite[Section 4.1.2]{DA-ZA4}, \cite[Section 2.5.1]{LI-RO1} and \cite[Section 1]{PES-ZA1}). In this paper we call $E$-valued $S$-Wiener process both genuine $E$-valued Wiener process and generalized $E$-valued Wiener process. In particular we call $E$-valued cylindrical Wiener process a generalized $E$-valued Wiener process with covariance operator $\Id_E$.
\end{rmk}

We remark that if Hypotheses \ref{hyp0} hold, then the right hand side of \eqref{mildF} is well defined. Indeed it is enough to show that the process 
\[
\{W_A(t)\}_{t\geq 0}:=\left\lbrace\int_0^te^{(t-s)A}Q^{\alpha}dW(s)\right\rbrace_{t\geq 0}
\]
is well defined. By \cite[Theorem 5.2 and Theorem 5.11]{DA-ZA4}, if Hypothesis \ref{hyp0}\eqref{hyp0.3} holds true, then $\{W_A(t)\}_{t\geq 0}$ is Gaussian, continuous in mean square and it has a continuous and predictable version (see \cite[Section 3.3]{DA-ZA4}). 


\begin{defi}\label{prog}
For $T>0$ and $p\geq 1$, we denote by $\X^p([0,T])$ the space of the progressively measurable (see \cite[Section 3.3]{DA-ZA4}) $\X$-valued processes $\{\psi(t)\}_{t\in[0,T]}$ such that
\[
\norm{\psi}^p_{\X^p([0,T])}:=\sup_{t\in [0,T]}\mathbb{E}\big[\norm{\psi(t)}^p\big]<+\infty.
\]
\end{defi}
We now state a general result in the theory of stochastic partial differential equations with Lipschitz continuous nonlinearities.
\begin{thm}\label{smild}
Let Hypotheses \ref{hyp0} hold. Let $T>0$ and let $\Phi:[0,T]\times\X\ra\X$ be measurable as a function from the $\sigma$-field $([0,T]\times\Omega\times\X,\mathcal{G}_T\times\mathcal{B}(\X))$ to $(\X,\mathcal{B}(\X))$, where $\mathcal{G}_T$ is the restriction to $[0,T]\times \Omega$ of the $\sigma$-field generated by the sets
\[(v,w]\times J,\;\; 0\leq v\leq w<+\infty,\;\; J\in\mathcal{F}_v.\] 
Assume that there exists $y\in \X$ such that the map $t\mapsto \Phi(t,y)$  from $[0,T]$ to $\X$ is $L^2$-summable, namely
\begin{align}\label{inty}
\int_0^T\|\Phi(t,y)\|^2dt<+\infty,
\end{align}
and that $\Phi$ is a Lipschitz continuous function on $\X$ uniformly with respect to $t\in [0,T]$, i.e. for every $x,y\in \X$ and $t\in [0,T]$, it holds
\begin{gather*}
\norm{\Phi(t,x)-\Phi(t,y)}\leq L_\Phi\norm{x-y};
\end{gather*}
where $L_\Phi>0$ is a constant independent of $t,x$ and $y$. Consider the stochastic partial differential equation
\begin{gather}\label{lalala}
\eqsys{
dX(t,x)=\big(AX(t,x)+\Phi(t,X(t,x))\big)dt+ Q^{\alpha}dW(t), & t\in(0,T];\\
X(0,x)=x\in \X.
}
\end{gather}
For each $x\in\X$, (\ref{lalala}) has unique mild solution $X(t,x)$ in $\X^2([0,T])$ such that
\begin{enumerate}[\rm(a)]
\item $X(\cdot,x)$ is $\mathbb{P}$-a.s. continuous in $[0,T]$;

\item The map $x\mapsto X(\cdot,x)$ from $\X$ to $\X^2([0,T])$ is Lipschitz continuous.
\end{enumerate}
\end{thm}
\noindent Condition \eqref{inty} is weaker than the one assumed in \cite[Theorem 7.5]{DA-ZA4}, namely there exists a constant $C_{F}>0$ such that, for every $t\in(0,T]$ and $x\in\X$, we have
\[\norm{\Phi(t,x)}\leq C_{F}(1+\norm{x}).\]
Instead \eqref{inty} is enough to prove the same results of \cite[Theorem 7.5]{DA-ZA4} which are used in this paper. The only difference is that the mild solution belongs to $\X^2([0,T])$ and not to all the $\X^p([0,T])$ space, with $p\geq 2$. We will give a proof of Theorem \ref{smild} in Appendix \ref{proofsmild} and we will give an example of a function $\Phi$ satisfying \eqref{inty}, but not satisfying the hypotheses of \cite[Theorem 7.5]{DA-ZA4}, in Section \ref{ex_1/2}.

If $F:\X\ra\X$ is Lipschitz continuous, then, by Theorem \ref{smild}, the transition semigroup
\[
P(t)\varphi(x):=\mathbb{E}[\varphi(X(t,x))],\qquad t\in[0,T],\ x\in\X,\ \varphi\in B_b(\X),
\]
is well defined, where $X(t,x)$ is the mild solution of \eqref{eqF}. Now we state a regularity result for the spatial derivative of the mild solution of \eqref{lalala}.

\begin{thm}\label{desmild}
In addition to the assumptions of Theorem \ref{smild}, assume $\Phi:[0,T]\times\X\ra\X$ is such that the map $x\mapsto \Phi(t,x)$ is Gateaux differentiable for every $t\in[0,T]$ and there exists $C>0$ such that for every $t\in[0,T]$ and $x,y\in\X$ it holds 
\[\|\J^G\Phi(t,x)y\|\leq C\|y\|.\]
Let $X(t,x)$ be the mild solution of \eqref{lalala}. Then the map $x\mapsto X(\cdot,x)$ from $\X$ to $\X^2([0,T])$ is Gateaux differentiable at $x_0\in\X$ with bounded and continuous directional derivatives. Moreover, for every $x_0,h\in\X$ and $t\in[0,T]$, the process $Y(t,h)=\J^G X(t,x_0) h$ is the unique mild solution of 
\[\eqsys{
dY(t,h)=(AY(t,h)+ (\J^G\Phi(t,X(t,x_0))Y(t,h))dt, & t\in(0,T];\\
 Y(0,h)=h\in\X.}
\]
\end{thm}

\noindent For a proof of Theorem \ref{desmild} we refer to \cite[Theorem 9.8]{DA-ZA4} and the arguments used in Appendix \ref{proofsmild}. We conclude the section by recalling a result that we will use in the next sections (see \cite[Lemma 2.3]{PES-ZA1}).

\begin{lemm}\label{rapphi}
Assume Hypotheses \ref{hyp0} hold true. Let $F\in C_{b}^2(\X;\X)$ and $\varphi\in C_{b}^2(\X)$. If $X(t,x)$ is the mild solution of \eqref{eqF} and $P(t)$ is the transition semigroup defined in \eqref{transition}, then for each $t\in(0,T]$, $P(t)\varphi\in C^2_{b}(\X)$ and
\begin{equation}\label{raaaa}
\varphi(X(t,x))=P(t)\varphi(x)+\int_0^t\scal{\D P(t-s)\varphi(X(s,x))}{Q^{\alpha}dW(s)},\quad \mathbb{P}\text{-a.s.}
\end{equation}
\end{lemm}

\vspace{0.3cm}
\begin{center}
\textsc{\small From here on, all the results involving processes must be understood as valid $\mathbb{P}$-a.s. for $t$ fixed.}
\end{center}
\vspace{0.3cm}

\section{Regularization results}\label{STRONGHI}

This section is devoted to the proof of Theorem \ref{Hstrong}. We start with some basic facts about the space $H_\alpha$.

\begin{prop}\label{propHalpha}
Assume that Hypotheses \ref{hyp0} hold true and let $H_{\alpha}:=Q^{\alpha}(\X)$. For every $h,k\in H_{\alpha}$ we set
\begin{align}\label{proscal}
\scal{h}{k}_\alpha:=\langle Q^{-\alpha}h,Q^{-\alpha}k\rangle.
\end{align}
Then $(H_{\alpha},\norm{\cdot}_\alpha)$ is a separable Hilbert space continuously embedded in $\X$, where $\norm{\cdot}_\alpha$ is the norm associated to the inner product in \eqref{proscal} and 
\begin{align}
\norm{h}\leq \|Q^{\alpha}\|_{\mathcal{L}(H_\alpha)}\norm{h}_\alpha,\label{Halpha1}
\end{align}
for every $h\in H_\alpha$. Furthermore the following holds
\begin{enumerate}[\rm(a)]

\item $Q^{\alpha}$ is linear and bounded from $H_\alpha$ to itself;\label{Halpha3}

\item $e^{tA}$ is a contraction semigroup in $H_\alpha$;\label{Halpha4}

\item $H_\alpha$ is dense in $\X$.\label{Halpha5}

\item $H_\alpha$ is a Borel subset of $\X$.\label{Halpha6}

\item $W_\alpha(t):=Q^{\alpha}W(t)$ is a $H_\alpha$-valued $Q^{2\alpha}$-Wiener process.\label{MBH}
\end{enumerate}
\end{prop}

\begin{proof}
Statements \eqref{Halpha3}-\eqref{Halpha6} are standard (e.g. \cite{DA-ZA1}) and their proofs are left to the reader.
Statement \eqref{MBH} follows noting that $Q^\alpha:\X\ra H_\alpha$ is continuous and the Borel subsets of $H_\alpha$ are Borel subsets of $\X$ (see \cite[Remark 5.1]{DA-ZA4}).
\end{proof}

\noindent We remark that if $\alpha=0$ then $H_\alpha=\X$. The study of the mild solution of \eqref{eqF} and of the transition semigroup \eqref{transition}, when $\alpha=0$, is already present in the literature, see for example \cite{DA-ZA1,FUR1}. Instead $H_{1/2}=Q^{1/2}(\X)$ is the Cameron--Martin space associated to the Gaussian measure with mean zero and covariance operator $Q$ on $\X$. This space is of fundamental interest for the Malliavin calculus, see for example \cite{BOGIE1,DA-ZA4}.
 
We conclude this introductory section with a lemma about a function that will be important throughout the rest of the paper.

\begin{lemm}\label{acuqa}
Let Hypotheses \ref{hyp1} hold true. For every $x\in\X$ and $t\in(0,T]$ the function $F_{x,t}: H_{\alpha}\rightarrow H_{\alpha}$ defined as
\[
F_{x,t}(h):=F(h+e^{tA}x),\qquad h\in H_{\alpha},
\]
is Lipschitz continuous, and 
\begin{align}
\int_0^T\|F_{x,s}(0)\|^2_\alpha ds<+\infty.\label{jj}
\end{align}
\end{lemm}

\begin{proof}
The Lipschitz continuity is an easy consequence of Hypothesis \ref{hyp1}. If $\alpha\in [1/4,1/2]$, condition \eqref{jj} follows by Hypothesis \ref{hyp1}\eqref{opl}. Instead, if $\alpha\in [0,1/4)$, by \cite[Proposition 2.1.1]{LUN1} and recalling that $A=-(1/2)Q^{2\alpha-1}$ and that $e^{sA}x$ belongs to $H_\alpha$ for every $s>0$ and $x\in\X$ (due to the analyticity of $e^{sA}$), we have 
\begin{align*}
\int_0^T\|F_{x,s}(0)\|^2_\alpha ds&=\int_0^T\|F(e^{sA}x)\|^2_\alpha ds\\
& \leq 2\max\set{L_{F,\alpha},\|F(0)\|_\alpha^2}\int_0^T(1+\|e^{sA}x\|^2_\alpha) ds\\ 
&= 2\max\set{L_{F,\alpha},\|F(0)\|_\alpha^2}\pa{T+ \int_0^T\|e^{sA}x\|^2_\alpha ds}\\
&= 2\max\set{L_{F,\alpha},\|F(0)\|_\alpha^2}\pa{ T+ \int_0^T\|Q^{-\alpha}e^{sA}x\|^2 ds}\\
&\leq 2\max\set{L_{F,\alpha},\|F(0)\|_\alpha^2}\pa{T+ \norm{x}^2\int_0^T\frac{C_\alpha}{t^{(2\alpha)/(1-2\alpha)}} ds}\\
&\leq 2\max\set{L_{F,\alpha},\|F(0)\|_\alpha^2}T\pa{1+ C_\alpha\frac{2\alpha-1}{4\alpha-1}T^{\frac{2\alpha}{2\alpha-1}}\norm{x}^2}<+\infty,
\end{align*}
for some positive constant $C_\alpha$.
\end{proof}

\subsection{Existence and uniqueness}\label{ExAndUni}

Now we want to show that Hypotheses \ref{hyp1} are sufficient to guarantee the existence and uniqueness of the mild solution of \eqref{eqF}. We remark that the existence and the uniqueness of the mild solution of equation \eqref{eqF}, if $F$ lacks continuity, was already studied in \cite{DA-FL1} and \cite{DA-FL2}, under a set of hypotheses that differ from ours. We cannot use the results seen in Section \ref{lip}, since $F$ is not a Lipschitz continuous function on $\X$. Instead, similarly to \cite{ES-ST1}, we take $H_\alpha$ as the underlying Hilbert space. For $\alpha\in [0,1/2)$ we stress that by Hypotheses \ref{hyp1}, Lemma \ref{propHalpha}\eqref{MBH} and the analyticity of the semigroup $e^{tA}$, the mild solution $X(t,x)$ of \eqref{eqF} belongs to $H_\alpha$ for $t\in(0,T]$, but not for $t=0$, because $X(0,x)=x\in\X$. Instead, for $\alpha=1/2$, we cannot state that $X(t,x)$ belongs to $H_\alpha$ not even for $t\in(0,T]$, because in this case the condition $e^{tA}\X\subseteq H_\alpha$ is not verified ($A=-(1/2)\Id$). Hence, in order to work on $H_\alpha$, it is necessary to define an auxiliary stochastic partial differential equation associated to \eqref{eqF} whose mild solution is a $H_\alpha$-valued process.

To do so we observe that, at least formally, the process $\{X(t,x)-e^{tA}x\}_{t\in[0,T]}$, for $x\in\X$, solves the equation
\begin{gather*}
\eqsys{
dY(t,0)=\big(AY(t,0)+F_{x,t}(Y(t,0))\big)dt+ Q^{\alpha}dW(t), & t\in(0,T];\\
Y(0,0)=0,
}
\end{gather*} 
and $Y(t,0)$ belongs to $H_\alpha$ for every $t\in[0,T]$. Indeed, still formally,
\begin{align*}
d(X(t,x)-e^{tA}x)&=\big(AX(t,x)+F(X(t,x))\big)dt+ Q^{\alpha}dW(t)-Ae^{tA}x\\
&=\big(A(X(t,x)-e^{tA}x)+F(X(t,x)+e^{tA}x-e^{tA}x)\big)dt+ Q^{\alpha}dW(t)\\
&=\big(A(X(t,x)-e^{tA}x)+F_{x,t}(X(t,x)-e^{tA}x)\big)dt+ Q^{\alpha}dW(t)
\end{align*}
and $X(0,x)-x=0$. This procedure was the main idea behind the techniques used in section. Indeed, let $T>0$ and for every $x\in\X$ we consider the stochastic partial differential equation
\begin{gather}\label{eqcameron}
\eqsys{
dZ_x(t,h)=\big(AZ_x(t,h)+F_{x,t}(Z_x(t,h))\big)dt+ Q^{\alpha}dW(t), & t\in(0,T];\\
Z_x(0,h)=h\in H_{\alpha},
}
\end{gather}
and its mild solution, namely the process $\{Z_x(t,h)\}_{t\geq 0}$ such that for $t\in [0,T]$,
\[
Z_x(t,h)=e^{tA}h+\int_0^te^{(t-s)A}F_{x,s}(Z_x(s,h))ds+\int_0^te^{(t-s)A}Q^{\alpha}dW(s).
\]
The reason to study the behaviour of the mild solution of \eqref{eqcameron} for every $h\in H_\alpha$, and not only for $h=0$, will became clear in Section \ref{SpacialRegularity}. In order to show that \eqref{eqcameron} has a unique mild solution we introduce the spaces $H^2_\alpha([0,T])$ defined as in Definition \ref{prog}, with $H_\alpha$ replacing $\X$, endowed with the norm 
\[
\norm{\psi}^2_{H_\alpha^2([0,T])}:=\sup_{t\in [0,T]}\mathbb{E}\big[\norm{\psi(t)}_\alpha^2\big].
\]

\begin{prop}\label{existenceZ}
Assume Hypotheses \ref{hyp1} hold true and let $x\in\X$. For each $h\in H_\alpha$, \eqref{eqcameron} has unique mild solution $Z_x(t,h)$ in $H^2_\alpha([0,T])$ such that
\begin{enumerate}[\rm(a)]
\item $Z_x(\cdot,h)$ is $\mathbb{P}$-a.s. continuous in $[0,T]$;

\item The map $h\mapsto Z_x(\cdot,h)$ from $H_\alpha$ to $H_\alpha^2([0,T])$ is Lipschitz continuous.
\end{enumerate}
\end{prop}

\begin{proof}
It is enough to observe that, by Lemma \ref{acuqa} the hypotheses of Theorem \ref{smild} for equation \eqref{eqcameron} are satified with $H_\alpha$ replacing $\X$.
\end{proof}

We are now ready to state and prove the main theorem of this subsection.

\begin{thm}\label{Hsmild}
If Hypotheses \ref{hyp1} hold true, then for every $x\in\X$ the stochastic partial differential equation \eqref{eqF} has a unique mild solution $X(t,x)$ belonging to $\X^2([0,T])$ and $\mathbb{P}$-a.s. path-continuous. 
Furthermore
\[X(t,x)=Z_x(t,0)+e^{tA}x,\]
where $Z_{x}(t,0)$ is the unique mild solution of \eqref{eqcameron} with $h=0$.
\end{thm}

\begin{proof}
We start by proving the uniqueness statement. Let $X(t,x)$ and $Y(t,x)$ be two mild solutions of  \eqref{eqF} in $\X^2([0,T])$. Then
\[X(t,x)=e^{tA}x+\int_0^te^{(t-s)A}F\big(X(s,x)\big)ds+\int_0^te^{(t-s)A}Q^{\alpha}dW(s),\]
and 
\begin{align*}
X(t,x)-e^{tA}x &=\int_0^te^{(t-s)A}F\big(X(s,x)+e^{sA}x-e^{sA}x\big)ds+\int_0^te^{(t-s)A}Q^{\alpha}dW(s)\\
&=\int_0^te^{(t-s)A}F_{x,s}\big(X(s,x)-e^{sA}x\big)ds+\int_0^te^{(t-s)A}Q^{\alpha}dW(s).
\end{align*}
Since $e^{(t-s)A}F_{x,s}(X(s,x)-e^{sA}x)$ and $\int_0^te^{(t-s)A}Q^{\alpha}dW(s)$ belong to $H_\alpha$, 
we get that the process $X(t,x)-e^{tA}x$ has values in $H_\alpha$. In the same way $Y(t,x)-e^{tA}x$ has values in $H_\alpha$. So $\{X(t,x)-Y(t,x)\}_{t\in[0,T]}$ is a $H_\alpha$-valued process. Observe that by Proposition \ref{propHalpha}\eqref{Halpha4}, we get
\begin{align*}
\E\sq{\norm{X(t,x)-Y(t,x)}_\alpha^2} &\leq T\E\sq{\int_0^t \norm{e^{(t-s)A}\big(F(X(s,x))-F(Y(s,x))\big)}^2_\alpha ds} \\ 
&\leq T L_{F,\alpha}\int_0^T\E\sq{\norm{X(s,x)-Y(s,x)}_\alpha^2}ds.
\end{align*}
Since $t\mapsto \E\sq{\norm{X(t,x)-Y(t,x)}_\alpha^2}$ is bounded in $[0,T]$, by the Gronwall inequality (see \cite[p. 188]{HEN1}) we get $\E[\norm{X(t,x)-Y(t,x)}_\alpha^2]=0$ for every $t\in[0,T]$ and $x\in\X$. We stress that we cannot make the same arguments in $\X^2([0,T])$, since the $F$ is $H_\alpha$-Lipschitz and may not be Lipschitz continuous on $\X$. By \eqref{Halpha1}, for every $p\geq 2$, $t\in[0,T]$ and $x\in\X$ 
\begin{align*}
\norm{X(t,x)-Y(t,x)}^2_{\X^2([0,T])} &=\sup_{t\in[0,T]}\E\big[\norm{X(t,x)-Y(t,x)}^2\big]\\
&\leq \|Q^{\alpha}\|_{\mathcal{L}(H_\alpha)}\sup_{t\in[0,T]}\E\big[\|X(t,x)-Y(t,x)\|_\alpha^2\big]=0.
\end{align*}
This concludes the proof of the uniqueness in $\X^2([0,T])$.

Now we show the existence of the mild solution. We have already noted that, when Hypotheses \ref{hyp1} hold, by Proposition \ref{existenceZ} the stochastic partial differential equation
\begin{gather*}
\eqsys{
dZ_x(t,0)=\big(AZ_x(t,0)+F_{x,t}(Z_x(t,0))\big)dt+ Q^{\alpha}dW(t), & t\in(0,T];\\
Z_x(0,0)=0,
}
\end{gather*}
has a unique mild solution $Z_x(t,0)$ in $H_\alpha^2([0,T])$. We claim that the process $X(t,x)=Z_x(t,0)+e^{tA}x$ is the mild solution of \eqref{eqF}. Indeed
\begin{align*}
X(t,x)&=e^{tA}x+Z_x(t,0)=e^{tA}x+\int_0^te^{(t-s)A}F_{x,s}\big(Z_x(s,0)\big)ds+\int_0^te^{(t-s)A}Q^{\alpha}dW(s)\\
&=e^{tA}x+\int_0^te^{(t-s)A}F\big(X(s,x)\big)ds+\int_0^te^{(t-s)A}Q^{\alpha}dW(s).
\end{align*}
Now we check that $e^{tA}x+Z_x(t,0)$ belongs to $\X^2([0,T])$. Indeed, by the contractivity of $e^{tA}$, Proposition \ref{propHalpha}\eqref{Halpha3} and \eqref{Halpha1} we have
\begin{align*}
\|e^{tA}x+Z_x(t,0)\|^2_{\X^2([0,T])} & =\sup_{t\in[0,T]}\E\big[\|e^{tA}x+Z_x(t,0)\|^p\big]\\
&= 2\|x\|^p+2\|Q^{\alpha}\|_{\mathcal{L}(H_\alpha)}\|Z_x(t,0)\|_{H_\alpha^2([0,T])}
\end{align*}
Since $\{Z_x(t,0)\}_{t\geq 0}$ belongs to $H_\alpha^2([0,T])$ by Proposition \ref{existenceZ}, the claim follows. The property $\mathbb{P}(\int_0^T\|e^{tA}x+Z_x(t,0)\|^2ds<+\infty)=1$ is an easy consequence of the fact that $e^{tA}x+Z_x(t,0)$ belongs to $\X^2([0,T])$. Therefore $X(t,x)=Z_x(t,0)+e^{tA}x$ is a mild solution of \eqref{eqF}. Furthermore, by \eqref{Halpha1}, $X(t,x)$ is $\mathbb{P}$-a.s. path-continuous. 
\end{proof}

\begin{rmk}\label{deriva}
In Theorem \ref{Hsmild} we have shown that for every $x\in\X$ and $t\in(0,T]$ the process $X(t,x)=Z_x(t,0)+e^{tA}x$ is the unique mild solution of \eqref{eqF} in $\X^2([0,T])$. So, for every $x\in\X$, $t\in(0,T]$ and $h\in H_\alpha$ the process $Z_{x+h}(t,0)+e^{tA}(x+h)$ is the unique mild solution of the stochastic partial differential equation 
\begin{gather}\label{eqxh}
\eqsys{
dX(t,x+h)=\big(AX(t,x+h)+F(X(t,x+h))\big)dt+ Q^{\alpha}dW(t), & t\in(0,T];\\
X(0,x+h)=x+h,
}
\end{gather}
that belongs to $\X^2([0,T])$. However in some cases it is more useful to represent the mild solution of \eqref{eqxh} by another process, as it will became apparent in the next subsection. For any $h\in H_{\alpha}$ and $x\in\X$, the process $Z_x(t,h)+e^{tA}x$ is the mild solution of \eqref{eqxh}. Indeed
\begin{align*}
Z_x(t,h)+e^{tA}x&=e^{tA}x+e^{tA}h+\int_0^te^{(t-s)A}F_{x,s}\big(Z_x(s,h)\big)ds+\int_0^te^{(t-s)A}Q^{\alpha}dW(s)\\
&=e^{tA}(x+h)+\int_0^te^{(t-s)A}F\big(Z_x(s,h)+e^{sA}x\big)ds+\int_0^te^{(t-s)A}Q^{\alpha}dW(s).
\end{align*}
In the same way, as in the proof of Theorem \ref{Hsmild}, it is possible to prove that $Z_x(t,h)+e^{tA}x$ belongs to $\X^2([0,T])$. So $X(t,x+h)=Z_x(t,h)+e^{tA}x$ almost surely with respect to $\mathbb{P}$.
\end{rmk}

\subsection{Space regularity}\label{SpacialRegularity}

In this subsection we will show that the mild solution of \eqref{eqF} constructed in Section \ref{ExAndUni} is Gateaux differentiable along $H_\alpha$. Now we clarify what we mean by ``differentiable along $H_\alpha$''.

\begin{defi}\label{derivataaaaa}
Let $Y$ be a Hilbert space endowed with the norm $\norm{\cdot}_Y$ and let $\Phi: \X\rightarrow Y$. 
\begin{enumerate}[\rm(i)]
\item We say that $\Phi$ is differentiable along $H_\alpha$ at the point $x\in\X$, if there exists $L\in\mathcal{L}(H_\alpha,Y)$ such that 
\begin{align*}
\lim_{\norm{h}_\alpha\ra 0}\frac{\norm{\Phi(x+h)-\Phi(x)-Lh}_Y}{\norm{h}_\alpha}=0.
\end{align*}
When it exists, the operator $L$ is unique and we set $\J_\alpha\Phi(x):=L$. If $Y=\R$, then $L\in H_\alpha^*$ and so there exists $k\in H_\alpha$ such that $Lh=\gen{h,k}_\alpha$ for any $h\in H_\alpha$. We set $\D_\alpha\Phi(x):=k$ and we call it $H_\alpha$-gradient of $\Phi$ at $x\in\X$. 

\item We say that $\Phi$ is two times differentiable along $H_\alpha$ at the point $x\in \X$ if it is differentiable along $H_\alpha$ at every point of $\X$ and there exists $T\in\mathcal{L}(H_\alpha,\mathcal{L}(H_\alpha,Y))$ such that
\begin{align*}
\lim_{\norm{k}_\alpha\ra 0}\frac{\norm{(\J_\alpha\Phi(x+k))h-(\J_\alpha\Phi(x))h-(Th)k}_{Y}}{\norm{k}_\alpha}=0.
\end{align*}
uniformly for $h\in H_\alpha$ with norm $1$. When it exists, the operator $T$ is unique and we set $\J^2_\alpha\Phi(x):=T$. If $Y=\R$, then $T\in\mathcal{L}(H_\alpha,H_\alpha^*)$, so there exists $S\in\mathcal{L}(H_\alpha)$ such that $(Th)(k)=\gen{Sh,k}_\alpha$, for any $h,k\in H_\alpha$. We set $\D^2_\alpha\Phi(x):=S$ and we call it $H_\alpha$-Hessian of $\Phi$ at $x\in\X$.

\item We say that $\Phi$ is Gateaux differentiable along $H_\alpha$ at the point $x\in\X$
if there exists $L\in\mathcal{L}(H_\alpha,Y)$ such that 
\begin{align*}
Lh=Y\text{-}\lim_{t\ra 0}\frac{\Phi(x+th)-\Phi(x)}{t}.
\end{align*}
When it exists, the operator $L$ is unique and we set $\J^G_\alpha\Phi(x):=L$.
\end{enumerate}
For simplicity sake we will write $\J_\alpha\Phi(x)h:=\J_\alpha\Phi(x)(h)$ and $\J^G_\alpha\Phi(x)h:=\J^G_\alpha\Phi(x)(h)$. For $k=1,2$, we denote by $C^k_{b,H_\alpha}(\X;Y)$ the set of the $k$-times differentiable functions along $H_\alpha$ such that the operator $\J_\alpha\Phi$, if $k=1$, and the operators $\J_\alpha\Phi$ and $\J^2_\alpha\Phi$, if $k=2$, are bounded. If $Y=\R$ we will simply write $C^k_{b,H_\alpha}(\X)$.
\end{defi}
\noindent We remark that if $\Phi:\X\ra Y$ is differentiable along $H_\alpha$ at $x\in\X$, then it is Gateaux differentiable along $H_\alpha$ at $x\in\X$ and it holds $\J_\alpha\Phi(x)=\J^G_\alpha\Phi(x)$. The derivative operators defined in Definition \ref{derivataaaaa}, are related to the ones presented in \cite{MAS2} and \cite{MAS1}. We will do a detailed comparison in Section \ref{Cmasi}.

We now prove some basic consequences of the above definition.

\begin{prop}\label{dalpha}
Let $Y$ be a Hilbert space endowed with the norm $\norm{\cdot}_Y$ and let $\Phi: \X\rightarrow Y$. If $\Phi$ is Fr\'echet differentiable at $x\in\X$, then it is differentiable along $H_\alpha$ and, for every $h\in H_\alpha$,
\begin{align}\label{J}
\J_\alpha\Phi(x)h=\J\Phi(x)h.
\end{align}
Furthermore if $\varphi:\X\ra\R$ is Fr\'echet differentiable at $x\in\X$, then we have
\begin{align*}
\D_\alpha\varphi(x)=Q^{2\alpha}\D\varphi(x).
\end{align*}
\end{prop}

\begin{proof}
By the Fr\'echet differentiability of $\Phi$ we know that for every $x\in\X$ 
\[\lim_{\norm{h}\rightarrow 0}\frac{\norm{\Phi(x+h)-\Phi(x)-\J\Phi(x)h}_Y}{\norm{h}}=0.\]
By \eqref{Halpha1} we have that $\norm{h}\ra 0$, whenever $\norm{h}_\alpha\ra 0$ and
\begin{align*}
0\leq &\lim_{\norm{h}_\alpha\rightarrow 0}\frac{\norm{\Phi(x+h)-\Phi(x)-\J\Phi(x)h}_Y}{\norm{h}_\alpha}\\
=& \lim_{\norm{h}_\alpha\rightarrow 0}\frac{\norm{\Phi(x+h)-\Phi(x)-\J\Phi(x)h}_Y}{\norm{h}}\frac{\norm{h}}{\norm{h}_\alpha}\\
\leq & \|Q^\alpha\|_{\mathcal{L}(H_\alpha)}\lim_{\norm{h}_\alpha\rightarrow 0}\frac{\norm{\Phi(x+h)-\Phi(x)-\J\Phi(x)h}_Y}{\norm{h}_\alpha}=0
\end{align*}
We stress that $\J\Phi(x)$ belongs to $\mathcal{L}(H_\alpha,Y)$. So we get that $\Phi$ is differentiable along $H_\alpha$ at $x$ and \eqref{J} holds. 
Moreover, for every $x\in\X$ and $h\in H_\alpha$, we have
\begin{align*}
\scal{\D_\alpha\varphi(x)}{h}_\alpha &= \J_\alpha\varphi(x)h=\J\varphi(x)h=\scal{\D\varphi(x)}{h}\\
&= \langle Q^{\alpha}\D\varphi(x), Q^{\alpha}h\rangle_\alpha=\langle Q^{2\alpha}\D\varphi(x), h\rangle_\alpha,
\end{align*}
hence $D_\alpha\varphi(x)=Q^{2\alpha}\D\varphi(x)$.
\end{proof}

\begin{lemm}\label{Falpha}
Assume that Hypotheses \ref{hyp1} hold true and let $k=1,2$. If $F$ belongs to $C^k_{b,H_\alpha}(\X;H_\alpha)$ then 
for any $x\in\X$ and $t\in(0,T]$ the function $F_{x,t}: H_{\alpha}\rightarrow H_{\alpha}$ defined as
\[F_{x,t}(h):=F(h+e^{tA}x),\qquad h\in H_{\alpha},\]
belongs $C^k_b(H_\alpha;H_\alpha)$. Furthermore $\|D_\alpha F_{x,t}\|_{\mathcal{L}(H_\alpha)}\leq L_{F,\alpha}$, where $L_{F,\alpha}$ is the $H_\alpha$-Lipschitz constant of $F$.
\end{lemm}

\begin{proof}
We only prove the statement in the case $k=1$, since the proof for $k=2$ is similar. By the definition of the space $C^1_{b,H_\alpha}(\X;H_\alpha)$ for every $y\in\X$
\begin{align}\label{con}
\lim_{\norm{h}_\alpha\ra 0}\frac{\norm{F(y+h)-F(y)-\J_\alpha F(y)h}_{\alpha}}{\norm{h}_{\alpha}}=0.
\end{align}
Now letting $y=e^{tA}x+h_0$ in \eqref{con} we get
\begin{align*}
\lim_{\norm{h}_\alpha\ra 0}\frac{\norm{F(e^{tA}x+h_0+h)-F(e^{tA}x+h_0)-\J_\alpha F(e^{tA}x+h_0)h}_{\alpha}}{\norm{h}_{\alpha}}=0.
\end{align*}
So $\J F_{x,t}(h_0)=\J_\alpha F(e^{tA}x+h_0)$. The furthermore part is an standard consequence of the identity we just showed.
\end{proof}

We are now ready to study the regularity of the mild solution of \eqref{eqF}.

\begin{thm}\label{deZ}
Assume that Hypotheses \ref{hyp1} hold true and let $F\in C^1_{b,H_{\alpha}}(\X;H_{\alpha})$. For every $x\in\X$ and $h\in H_\alpha$ let $Z_x(t,h)$ be the mild solution of \eqref{eqcameron}. The problem 
\begin{gather}\label{eqdeZ}
\eqsys{
dY(t,h_0)=\big(AY(t,h_0)+\J F_{x,t}\big(Z_x(t,h)\big)Y(t,h_0)\big)dt, & t\in(0,T];\\
Y(0,h_0)=h_0\in H_\alpha,
}
\end{gather}
admits a unique mild solution $Y(\cdot,h_0)$ in $H_{\alpha}^2([0,T])$. Furthermore for every $t\in [0,T]$,
\begin{equation}\label{stigro}
\norm{Y(t,h_0)}_{\alpha}\leq e^{TL_{F,\alpha}}\norm{h_0}_{\alpha}.
\end{equation}
Finally the map $h\mapsto Z_x(\cdot,h)$ is Gateaux differentiable with values in $H^2_\alpha([0,T])$ and for any $t\in[0,T]$ and $h_0\in H_\alpha$ it holds $Y(t,h_0)=\J^G Z_x(t,h)h_0$.
\end{thm}

\begin{proof}
Consider the linear operator $V$ defined on $H_{\alpha}^2([0,T])$ as
\[V(Y)(t):=e^{tA}h_0+\int_0^te^{(t-s)A}\J F_{x,t}\big(Z_x(s,h)\big)Y(s)ds,\qquad t\in [0,T].\]
We want to apply the contraction mapping theorem to $V$, since a fixed point of $V$ is a mild solution of \eqref{eqdeZ}. First we check that $V(H_{\alpha}^2([0,T]))\subseteq H_{\alpha}^2([0,T])$. If $Y\in H^2_\alpha([0,T])$, then by Proposition \ref{propHalpha}\eqref{Halpha4} and Lemma \ref{Falpha} we have by standard computations
\begin{align}
\norm{V(Y)}^2_{H^2_\alpha([0,T])}
&\leq 2\norm{h_0}_\alpha^2+2T^2L_{F,\alpha}^2\norm{Y}^2_{H^2_\alpha([0,T])}<+\infty.\label{cuffie}
\end{align}
Now we show that $V$ is Lipschitz continuous on $H_\alpha^2([0,T])$. Let $Y_1,Y_2\in H_{\alpha}^2([0,T])$, then by standard calculations we get
\begin{align*}
\norm{V(Y_1)-V(Y_2)}^2_{H_{\alpha}^2[0,T]}
\leq T\sup_{t\in [0,T]}\left(\mathbb{E}\left[\left(\int_0^t\norm{e^{(t-s)A}\J F_{x,s}(Z_x(s,h))(Y_1(s)-Y_2(s))}_{\alpha}ds\right)^2\right]\right)
\end{align*}
Using the same arguments as in \eqref{cuffie} we obtain
\begin{align*}
\norm{V(Y_1)-V(Y_2)}^2_{H_{\alpha}^2([0,T])}\leq T^2L^2_{F,\alpha}\norm{Y_1-Y_2}_{H^2_{\alpha}([0,T])}^2.
\end{align*}
So there exists $T^*>0$ such that $V$ is a contraction on $H^2_\alpha([0,T^*])$. 
\[Y(t):=\eqsys{Y_r(y), & t\in [rT^*,(r+1)T^*],\ r=0,\ldots,n,\\
Y_n(t), & t\in [nT^*,T];}\]
By a classical arguments we have that $Y$ is the unique mild solution of \eqref{eqdeZ} in $H^2_\alpha([0,T])$.

To prove \eqref{stigro} we start by observing that by Proposition \ref{propHalpha}\eqref{Halpha4} and Lemma \ref{Falpha}
\begin{align*}
\norm{Y(t,h_0)}_{\alpha}
&\leq \norm{h_0}_{\alpha}+L_{F,\alpha}\int_0^t\norm{Y(s,h_0)}_{\alpha}ds.
\end{align*}
Recalling that the functions $Y(\cdot,h_0)$, $Z_x(\cdot,h)$ and $\J F_{x,t}$ are continuous, the Gronwall inequality yields \eqref{stigro}.

By Proposition \ref{existenceZ} and Lemma \ref{Falpha}, if $F\in C^1_{b,H_{\alpha}}(\X;H_{\alpha})$, then for each $T>0$, the map $h\mapsto Z_x(\cdot,h)$ from $H_\alpha$ to $H_{\alpha}^2([0,T])$ is Gateaux differentiable with bounded and continuous directional derivatives, for every $x\in\X$. Moreover, for any $h_0\in H_\alpha$, the process $Y(t,h_0):=\J^G Z_x(t,h)h_0$ is the unique mild solution in $H_\alpha([0,T])$ of \eqref{eqdeZ}. 
\end{proof}

Now we want to study the process $\J^G_{\alpha} X(t,x)h_0$ with $x\in\X$ and $h_0\in H_{\alpha}$.

\begin{thm}\label{deX}
Assume that Hypotheses \ref{hyp1} hold true, and let $F\in C^1_{b,H_{\alpha}}(\X;H_{\alpha})$. The map $x\mapsto X(\cdot,x)$ from $\X$ to $\X^2([0,T])$ is Gateaux differentiable along $H_\alpha$ and for $x\in\X$, $t\in[0,T]$ and $h\in H_{\alpha}$ it holds
\begin{gather}\label{DX=DZ}
\J^G_{\alpha}X(t,x)h=\J^G Z_x(t,0)h.
\end{gather}
\end{thm}

\begin{proof}
Let $x\in\X$, $h\in H_{\alpha}$, $t\in[0,T]$ and $s>0$. By Remark \ref{deriva} we know that $X(t,x+sh)=Z_x(t,sh)+e^{tA}x$, so $X(t,x+sh)-X(t,x)=Z_x(t,sh)-Z_x(t,0)\in H_{\alpha}$. Hence by \eqref{Halpha1} and Proposition \ref{existenceZ} we have
\begin{align*}
0\leq &\lim_{s\ra 0}\norm{\frac{X(\cdot,x+sh)-X(\cdot,x)}{s}-\J^G Z_x(\cdot,0)h}^2_{\X^2([0,T])}\\
=&\lim_{s\ra 0}\sup_{t\in[0,T]}\E\Bigg[\norm{\frac{X(t,x+sh)-X(t,x)}{s}-\J^G Z_x(t,0)h}^2\Bigg]\\
=&\lim_{s\ra 0}\sup_{t\in[0,T]}\E\Bigg[\norm{\frac{Z_x(t,sh)-Z_x(t,0)}{s}-\J^G Z_x(t,0)h}^2\Bigg]\\
\leq&\|Q^{\alpha}\|_{\mathcal{L}(H_\alpha)}\lim_{s\ra 0}\sup_{t\in[0,T]}\E\Bigg[\norm{\frac{Z_x(t,sh)-Z_x(t,0)}{s}-\J^G Z_x(t,0)h}_\alpha^2\Bigg]\\
=& \|Q^{\alpha}\|_{\mathcal{L}(H_\alpha)}\lim_{s\ra 0}\norm{\frac{Z_x(\cdot,sh)-Z_x(\cdot,0)}{s}-\J^G Z_x(t,0)h}^2_{H^2_\alpha([0,T])}=0.
\end{align*}
Linearity and continuity in $H_\alpha$ of $h\mapsto\J^G_{\alpha}X(t,x)h$ follows from the linearity and continuity of $h\mapsto \J^G Z_x(t,0)h$.
\end{proof}

To end this subsection we state and prove the following corollary.

\begin{cor}
Assume that Hypotheses \ref{hyp1} hold true, let $T>0$ and let $F\in C^1_{b,H_{\alpha}}(\X;H_{\alpha})$. If $g:\X\rightarrow\R$ is a function belonging to $C^1_{b,H_{\alpha}}(\X)$ and $h\in H_{\alpha}$, then for any $x\in X$ and $t\in [0,T]$
\begin{gather}\label{accatena}
((\J^G_\alpha(g\circ X))(t,x))h=\scal{\left(\D_{\alpha}g\right) (X(t,x))}{\J^G_\alpha X(t,x)h}_{\alpha}.
\end{gather}
\end{cor}

\begin{proof}
Since $g\in C^1_{b,H_{\alpha}}(\X)$, then for every $x\in\X$ and $h\in H_{\alpha}$
\[g(x+\epsilon h)=g(x)+\eps \scal{\D_{\alpha}g(x)}{h}_{\alpha}+o(\eps)\qquad \eps\ra 0.   \]
We define for $x\in\X$, $h\in H_{\alpha}$, $t\in [0,T]$ and $\eps>0$
\[K_{\eps}(t,x,h):=X(t,x+\eps h)-X(t,x)-\eps\J^G Z_x(t,0)h=Z_x(t,h)-Z_x(t,0)-\eps\J^G Z_x(t,0)h.\]
Observe that by Proposition \ref{existenceZ}, we have that $\norm{K_{\eps}(\cdot,x,h)}^2_{H_{\alpha}^2
([0,T])}=o(\eps)$, when $\eps$ goes to zero. Hence for $\eps\ra 0$
\begin{align*}
g\big(X(t,x+\eps h)\big)& =g\big(X(t,x)+\eps\J^G Z_x(t,0)h+K_{\eps}(t,x,h)\big)\\
& =g\big(X(t,x)+\eps(\J^G Z_x(t,0)h+\eps^{-1}K_{\eps}(t,x,h)\big)\\
& =g(X(t,x))+\eps\scal{(\D_{\alpha}g)(X(t,x))}{\J^G Z_x(t,0)h+\eps^{-1}K_{\eps}(t,x,h)}_\alpha+o(\eps)\\
& =g(X(t,x))+\eps\scal{(\D_{\alpha}g)(X(t,x))}{\J^G Z_x(t,0)h}_\alpha\\
&\phantom{aaaaaaaaaaaaaaaaaaaaaaaaaaaa}+\scal{(\D_{\alpha}g)(X(t,x))}{K_{\eps}(t,x,h)}_\alpha+o(\eps).
\end{align*}
So for $\eps\ra 0$ we get
\begin{align*}
0\leq &\E\sq{\abs{g\big(X(t,x+\eps h)\big)-g(X(t,x))-\eps\scal{(\D_{\alpha}g)(X(t,x))}{\J^G Z_x(t,0)h}_\alpha}^2}\\
\leq &\sup_{t\in[0,T]}\E\sq{\abs{g\big(X(t,x+\eps h)\big)-g(X(t,x))-\eps\scal{(\D_{\alpha}g)(X(t,x))}{\J^G Z_x(t,0)h}_\alpha}^2}\\
= &\sup_{t\in[0,T]}\E\Big[\abs{\scal{(\D_{\alpha}g)(X(t,x))}{K_{\eps}(t,x,h)}_\alpha}^2\Big]+o(\eps)\\
\leq &\pa{\sup_{x\in\X}\norm{\J_\alpha g(x)}_{\mathcal{L}(H_\alpha)}}\pa{\sup_{t\in[0,T]}\E\Big[\norm{K_{\eps}(t,x,h)}_\alpha^2\Big]}+o(\eps)\\
=& \pa{\sup_{x\in\X}\norm{\J_\alpha g(x)}_{\mathcal{L}(H_\alpha)}}\norm{K_{\eps}(\cdot,x,h)}_{H_\alpha^2([0,T])}^p+o(\eps)=o(\eps)
\end{align*}
This imples that $\mathbb{P}$-a.s it holds $((\J^G_\alpha(g\circ X))(t,x))h=\scal{\left(\D_{\alpha}g\right) (X(t,x))}{\J^G Z_x(t,0)h}_{\alpha}$ and the proof is concluded recalling Theorem \ref{deX}.
\end{proof}

\subsection{Proof of Theorem \ref{Hstrong}}\label{StrongFellerProperty}

Throughout this subsection $X(t,x)$ will denote the mild solution of \eqref{eqF}, while $P(t)$ is its associated transition semigroup, defined in \eqref{transition}. To prove Theorem \ref{Hstrong} we will use a similar procedure to the one used in \cite[section 7.7]{DA-ZA1} and \cite{PES-ZA1}. First we are going to prove Theorem \ref{Hstrong} for sufficiently regular functions $F$ and $\varphi$. Note that Lemma \ref{FLemma} is an adaptation of Lemma \ref{rapphi} to our situation.

\begin{lemm}\label{FLemma}
Assume Hypotheses \ref{hyp1} hold true. Let $F\in C_{b,H_{\alpha}}^2(\X;H_{\alpha})$, $\varphi\in C_{b,H_{\alpha}}^2(\X)$ and $x\in\X$. Then for each $t\in(0,T]$, $P(t)\varphi\in C^2_{b,H_{\alpha}}(\X)$ and 
\begin{equation}\label{Feq}
\varphi(X(t,x))=P(t)\varphi(x)+\int_0^t\scal{\D_{\alpha} P(t-s)\varphi(X(s,x))}{Q^{\alpha}dW(s)}_{\alpha}.
\end{equation}
\end{lemm}

\begin{proof}
Let $x\in\X$ and consider the transition semigroup
\[T_x(t)\psi(h):=\mathbb{E}[\psi(Z_x(t,h))],\quad t\in[0,T],\ h\in H_\alpha,\ \psi\in B_b(H_{\alpha});\]
where $Z_x(t,h)$ is the mild solution of \eqref{eqcameron}.  Let $\varphi\in C_{b,H_{\alpha}}^2(\X)$ and consider the function $\widehat{\varphi}(h):=\varphi(e^{tA}x+h)$ on $H_{\alpha}$. Proceeding in the same way as in Lemma \ref{Falpha} we have $\widehat{\varphi}\in C_{b}^2(H_{\alpha})$. Moreover since by Theorem \ref{Hsmild} it holds $X(t,x)=e^{tA}x+Z_x(t,0)$ then
\[P(t)\varphi(x)=T_x(t)\widehat{\varphi}(0).\]
We recall that, by Lemma \ref{rapphi} and Lemma \ref{Falpha}, $T_x(t)\hat{\varphi}\in C_b^2(H_\alpha)$. Moreover, by Remark \ref{deriva}, if $x\in\X$, $h\in H_{\alpha}$ and $t\in[0,T]$ then 
\begin{equation*}
P(t)\varphi(x+h)=\mathbb{E}\big[\varphi(e^{tA}x+Z_x(t,h))\big]=T_x(t)\widehat{\varphi}(h).
\end{equation*}
We claim that $P(t)\varphi$ is differentiable along $H_\alpha$. Indeed for every $x\in \X$
\begin{align}
&\lim_{\norm{h}_\alpha\ra0}\frac{\abs{P(t)\varphi(x+h)-P(t)\varphi(x)-\gen{D T_x(t)\hat{\varphi}(0),h}_\alpha}}{\norm{h}_\alpha}\notag\\
=&\lim_{\norm{h}_\alpha\ra0}\frac{\abs{T_x(t)\hat{\varphi}(h)-T_x(t)\hat{\varphi}(0)-\gen{D T_x(t)\hat{\varphi}(0),h}_\alpha}}{\norm{h}_\alpha}=0.\label{concam}
\end{align}
So $P(t)\varphi$ belongs to $C^1_{b,H_\alpha}(\X)$. A similar argument gives $P(t)\varphi\in C^2_{b,H_{\alpha}}(\X)$.
By Lemma \ref{rapphi}, for each $t\in (0,T]$, $x\in\X$ and $h\in H_{\alpha}$ we have
\begin{equation}\label{lemcam}
\widehat{\varphi}(Z_x(t,h))=T_x(t)\widehat{\varphi}(h)+\int_0^t\scal{\D T_x(t-s)\widehat{\varphi}(Z_x(s,h))}{Q^{\alpha}dW(s)}_{\alpha}.
\end{equation}
So \eqref{Feq} follows by \eqref{lemcam} with $h=0$ and \eqref{concam}. 
\end{proof}

Now we prove a variant of the Bismut--Elworthy--Li formula. 

\begin{prop}\label{BEL}
Assume that Hypotheses \ref{hyp1} hold. Let $F\in C_{b,H_{\alpha}}^2(\X;H_{\alpha})$ and $\varphi\in C_{b,H_{\alpha}}^2(\X)$. For every $x\in\X$, $h\in H_{\alpha}$ and $t\in(0,T]$
\begin{align}\label{FDeq}
\scal{\D_{\alpha} P(t)\varphi(x)}{h}_{\alpha}=\frac{1}{t}\mathbb{E}\left[\varphi(X(t,x))\int_0^t\scal{\J^G_{\alpha} X(s,x)h}{Q^{\alpha}dW(s)}_{\alpha}\right].
\end{align}
Furthermore
\begin{align}\label{FEsti}
\vert \scal{\D_{\alpha} P(t)\varphi(x)}{h}_{\alpha}\vert^2 \leq \frac{1}{t^2}\Vert\varphi\Vert^2_{\infty}\mathbb{E}\left[\int_0^t\norm{\J^G_{\alpha} X(s,x)h}_{\alpha}^{2}ds\right].
\end{align}
\end{prop}

\begin{proof}
(\ref{FEsti}) is a standard consequence of (\ref{FDeq}) and the It\^o isometry (see \cite[Lemma 3.1.5]{OKS1}) so we will only show (\ref{FDeq}). We recall that, by Theorem \ref{deX}, $\J^G_{\alpha}X(t,x)h=\J Z_x(t,0)h$. Let $h\in H_{\alpha}$, $t\in(0,T]$ and $x\in\X$. Multiplying both sides of (\ref{Feq}) by
\[
\int_0^t\scal{\J^G Z_x(s,0)h}{Q^{\alpha}dW(s)}_{\alpha},
\] 
and taking the expectations we get
\begin{align*}
&\mathbb{E}\bigg[\varphi(X(t,x))\int_0^t \scal{\J^G Z_x(s,0)h}{Q^{\alpha}dW(s)}_{\alpha}\bigg]\\
=&\mathbb{E}\sq{P(t)\varphi(x)\int_0^t\scal{\J^G Z_x(s,0)h}{Q^{\alpha}dW(s)}_{\alpha}}\\
+&\mathbb{E}\sq{\int_0^t\scal{\D_{\alpha} P(t-s)\varphi(X(s,x))}{Q^{\alpha}dW(s)}_{\alpha}\int_0^t\scal{\J^G Z_x(s,0)h}{Q^{\alpha}dW(s)}_{\alpha}}.
\end{align*}
We recall that the process $\{Q^{\alpha}W(s)\}_{s\geq 0}$ is a $H_{\alpha}$-valued Wiener process (see Proposition \ref{propHalpha}\eqref{MBH}). By \cite[Remark 2]{EL-LI1}, the process $\{\int_0^t\scal{\J^G Z_x(s,0)h}{Q^{\alpha}dW(s)}_{\alpha}\}_{t\geq 0}$ is a martingale provided that for every $t\in[0,T]$ and $h\in H_\alpha$
\[\int_0^t\mathbb{E}\big[\Vert \J^G Z_x(s,0)h\Vert_{\alpha}^2\big]ds<+\infty.\]
By Theorem \ref{deZ}, we know that $\J^G Z_x(\cdot,0)h\in H^2_{\alpha}([0,T])$, then for any $t\in[0,T]$ and $h\in H_\alpha$
\[\int_0^t\mathbb{E}\Big[\Vert \J^G Z_x(s,0)h\Vert_{\alpha}^2\Big]ds\leq T\|\J^G Z_x(\cdot,0)h\|^2_{H_\alpha^2([0,T])}<+\infty.\] 
Hence $t\mapsto\int_0^t\langle\J^G Z_x(s,0)h,Q^{\alpha}dW(s)\rangle_{\alpha}$ is a martingale and we have for every $t\in[0,T]$, $x\in\X$ and $h\in H_\alpha$
\[\mathbb{E}\sq{P(t)\varphi(x)\int_0^t\scal{\J^G Z_x(s,0)h}{Q^{\alpha}dW(s)}_{\alpha}}=0.\]
Hence by \eqref{accatena}, with $G=P(t-s)\varphi$, and the It\^o isometry we obtain
\begin{align*}
&\mathbb{E}\sq{\int_0^t\scal{(\D_{\alpha} P(t-s)\varphi)(X(s,x))}{Q^{\alpha}dW(s)}_{\alpha}\int_0^t\scal{\J^G Z_x(s,0)h}{Q^{\alpha}dW(s)}_{\alpha}}\\
=&\mathbb{E}\sq{\int_0^t\scal{(\D_{\alpha} P(t-s)\varphi)(X(s,x))}{\J^G Z_x(s,0)h}_{\alpha}ds}\\
=&\mathbb{E}\Bigg[\int_0^t \J^G_\alpha(((P(t-s)\varphi)\circ X)(s,x))h ds\Bigg]\\
=&\int_0^t \big(\J^G_{\alpha}\mathbb{E}\big[(P(t-s)\varphi\circ X)(s,x)\big]\big)hds.
\end{align*}
By the very definition of $P(t)$ we know that $\mathbb{E}[(P(t-s)\varphi\circ X)(s,x)]=(P(s)P(t-s)\varphi)(x)=P(t)\varphi(x)$. Recalling that $P(t)\varphi$ belongs to $C^2_{b,H_\alpha}(\X)$ it holds $\J^G_\alpha P(t)\varphi(x)=\J_\alpha P(t)\varphi(x)$. So, by Lemma \ref{FLemma}, we conclude
\begin{align*}
\mathbb{E}\sq{\varphi(X(t,x))\int_0^t\scal{\J^G Z_x(s,0)h}{Q^{\alpha}dW(s)}_{\alpha}} &=\int_0^t\scal{\D_{\alpha} P(t)\varphi(x)}{h}_{\alpha}ds\\
&=t\scal{\D_{\alpha} P(t)\varphi(x)}{h}_{\alpha}.
\end{align*}
Recalling \eqref{DX=DZ} we get the thesis.
\end{proof}

The last step before proving Theorem \ref{Hstrong} is the following corollary.

\begin{cor}
Assume that Hypotheses \ref{hyp1} hold. Let $F\in C_{b,H_{\alpha}}^2(\X;H_{\alpha})$ and $\varphi\in C_{b,H_{\alpha}}^2(\X)$. For every $t\in(0,T]$, $x\in\X$ and $h\in H_{\alpha}$ 
\begin{equation}\label{HlipC}
\vert P(t)\varphi(x+h)-P(t)\varphi(x)\vert\leq \frac{e^{L_{F,\alpha} T}}{\sqrt{t}}\Vert\varphi\Vert_{\infty}\norm{h}_{\alpha}.
\end{equation}
\end{cor}

\begin{proof}
Taking into account \eqref{stigro}, \eqref{DX=DZ} and \eqref{FEsti} we obtain the gradient estimate 
\begin{align}\label{stigra}
\Vert \D_{\alpha} P(t)\varphi(x)\Vert_{\alpha}\leq \frac{e^{L_{F,\alpha} T}}{\sqrt{t}}\Vert\varphi\Vert_{\infty},\qquad t\in(0,T],\ x\in\X.
\end{align}
Let $x\in\X$ and $h\in H_{\alpha}$, then by the mean value theorem, there exists $c_h>0$ such that
\[P(t)\varphi(x+h)-P(t)\varphi(x)=\scal{\D_{\alpha} P(t)\varphi(x+c_hh)}{h}.\]
So, by \eqref{stigra}, the thesis follows.
\end{proof}

Now we can prove Theorem \ref{Hstrong}.

\begin{proof}[of Theorem \ref{Hstrong}]
We start by assuming that  $F\in C_{b,H_{\alpha}}^2(\X;H_{\alpha})$ and we show that since \eqref{HlipC} is verified for $\varphi\in C^2_{b,H_{\alpha}}(\X)$ then it also holds for $\varphi\in B_{b}(\X)$. We recall that by \cite[Theorem 5.4]{ZAB1}, if $\varphi\in C_b(\X)$ then there exists a sequence $\{\varphi_n\}_{n\in\N}\subseteq C^2_b(\X)$ such that, for every $x\in\X$, 
\[\lim_{n\rightarrow+\infty}\varphi_n(x)=\varphi(x),\qquad\qquad\norm{\varphi_n}_{\infty}\leq \norm{\varphi}_{\infty}.\] 
Since $C^2_b(\X)\subseteq C^2_{b,H_\alpha}(\X)$ by Proposition \ref{dalpha}, \eqref{HlipC} yields 
\begin{equation*}
\vert P(t)\varphi_n(x+h)-P(t)\varphi_n(x)\vert\leq \frac{e^{L_{F,\alpha}T}}{\sqrt{t}}\Vert\varphi_n\Vert_{\infty}\norm{h}_{\alpha},\qquad n\in\N,\ t\in(0,T],\ h\in H_\alpha.
\end{equation*}
Observe that by the dominated convergence theorem $P(t)\varphi_n(x+h)$ and $P(t)\varphi_n(x)$ converge to $P(t)\varphi(x+h)$ and $P(t)\varphi(x)$, respectively. Therefore \eqref{HlipC} is verified also for $\varphi\in C_b(\X)$. By the Riesz representation theorem and \eqref{HlipC}, for every $x\in\X$, $h\in H_{\alpha}$ and $t\in(0,T]$, we have the following estimate for the total variation of the finite measure $\mathscr{L}(X(x+h,t))-\mathscr{L}(X(x,t))$
\begin{align*}
{\rm Var}\big(\mathscr{L}(X(t,x+h))-\mathscr{L}(X(t,x))\big):=&\sup_{\substack{\varphi\in C_b(\X)\\ \norm{\varphi}_{\infty}\leq 1}}\abs{\int_{\X}\varphi d\bigg (\mathscr{L}\big(X(t,x+h)\big)-\mathscr{L}\big(X(t,x)\big)\bigg)}\\
=&\sup_{\substack{\varphi\in C_b(\X)\\ \norm{\varphi}_{\infty}\leq 1}}\vert P(t)\varphi(x+h)-P(t)\varphi(x)\vert\leq \frac{e^{L_{F,\alpha}T}}{\sqrt{t}}\norm{h}_{\alpha}.
\end{align*}
Let $\varphi\in B_b(\X)$, then for $t\in(0,T]$, $x\in\X$ and $h\in H_\alpha$
\begin{align*}
\vert P(t)\varphi(x+h)-P(t)\varphi(x)\vert 
&=\abs{\int_{\X}\varphi d\Big(\mathscr{L}\big(X(t,x+h)\big)-\mathscr{L}\big(X(t,x)\big)\Big)}\\
&\leq \Vert\varphi\Vert_{\infty}\frac{e^{L_{F,\alpha}T}}{\sqrt{t}}\norm{h}_{\alpha}.
\end{align*}
As a second step, we prove that \eqref{HlipC} is verified for $\varphi\in C^2_{b,H_{\alpha}}(\X)$ if $F$ satisfies Hypotheses \ref{hyp1}. We recall that by Lemma \ref{acuqa}, $F_{x,t}$ is Lipschitz continuous on $H_{\alpha}$, so it is possible to construct a sequence $\{ F^{(n)}_{x,t} \}_{n\in\N}\subseteq C_b^2(H_{\alpha};H_{\alpha})$ (see \cite[Lemma 2.5]{PES-ZA1}) such that the functions $F^{(n)}_{x,t}$ are Lipschitz continuous with Lipschitz constant less or equal than $L_{F,\alpha}$, and
\[\lim_{n\rightarrow+\infty}\|F^{(n)}_{x,t}(h)-F_{x,t}(h)\|_{\alpha}=0,\qquad h\in H_{\alpha}.\]
We consider the transitions semigroups 
\[P^{(n)}(t)\varphi(x):=\mathbb{E}\big[\varphi(X^{(n)}(t,x))\big],\qquad \varphi\in C_b(\X),\]
where $X^{(n)}(t,x):=Z^{(n)}_x(t,0)+e^{tA}x$ and $Z^{(n)}_x(t,0)$ is the mild solution of 
\begin{gather*}
\eqsys{
dZ_x(t,0)=\big(AZ_x(t,0)+F^{(n)}_{x,t}(Z_x(t,0))\big)dt+ Q^{\alpha}dW(t), & t\in(0,T];\\
Z_x(0,0)=0.
}
\end{gather*}
Fix $\varphi\in C^2_{b,H_{\alpha}}(\X)$. Then by \eqref{HlipC} for every $x\in\X$, $h\in H_{\alpha}$ and $t\in (0,T]$,  we get
\[\vert P^{(n)}(t)\varphi(x+h)-P^{(n)}(t)\varphi(x)\vert\leq \frac{e^{L_{F,\alpha}T}}{\sqrt{t}}\Vert\varphi\Vert_{\infty}\norm{h}_{\alpha}.\]
By \cite[Theorem A.1]{PES-ZA1} there exists a subsequence $\{Z^{(n_k)}_x(t,0)\}_{k\in\N}$ such that 
\[X^{(n_k)}(t,x)=Z^{(n_k)}_x(t,0)+e^{tA}x\rightarrow Z_x(t,0)+e^{tA}x=X(t,x),\]
where the convergence is almost surely with respect to $\mathbb{P}$. Since $\varphi$ is bounded and continuous then
\begin{align*}
P^{(n_k)}(t)\varphi(x)&=\E\Big[\varphi(X^{(n_k)}(t,x))\Big]=\E\sq{\varphi(Z^{(n_k)}_x(t,0)+e^{tA}x)}\\
&\rightarrow \E\sq{\varphi(Z_x(t,0)+e^{tA}x)}=\E\big[\varphi(X(t,x))\big]=P(t)\varphi(x).
\end{align*}
So for every $x\in\X$, $h\in H_{\alpha}$ and $t\in (0,T]$,
\[\vert P(t)\varphi(x+h)-P(t)\varphi(x)\vert\leq \frac{e^{L_{F,\alpha}T}}{\sqrt{t}}\Vert\varphi\Vert_{\infty}\norm{h}_{\alpha}.\]
By the first step we conclude the proof.
\end{proof}

\begin{rmk}\label{x_dip}
We stress that the $H_\alpha$-Lipschitzianity of $F$ in Hypotheses \ref{hyp1} can be replaced by a weaker condition: for every $x\in\X$ there exists $L_{F,\alpha}(x)>0$ such that for every $x\in\X$ and $h\in H_\alpha$
\[\norm{F(x+h)-F(x)}_\alpha\leq L_{F,\alpha}(x)\norm{h}_\alpha.\]
Clearly, with this condition, whenever the costant $L_{F,\alpha}$ appears in the paper it has to be replaced with $L_{F,\alpha}(x)$. So the semigroup $P(t)$ does not map $B_b(\X)$ in $\lip_{b,H_\alpha}(\X)$, but for every $\varphi\in B_b(\X)$, we have that for every $t\in(0,T]$, $x\in\X$ and $h\in H_\alpha$,
\[\abs{P(t)\varphi(x+h)-P(t)\varphi(x)}\leq \frac{e^{L_{F,\alpha}(x)T}}{\sqrt{t}}\norm{\varphi}_{\infty}\norm{h}_\alpha.\]
\end{rmk}

\section{Proof of Theorem \ref{Hstrongvar}}\label{proofH}

This section is devoted to the proof of Theorem \ref{Hstrongvar}. First of all we stress that $F$ is Lipschitz continuos, since $Q^{-\alpha}F$ is Lipschitz continuous. Indeed, let $x,y\in\X$, we have
\begin{align*}
\norm{F(x)-F(y)}&=\norm{Q^\alpha Q^{-\alpha}(F(x)-F(y))}\\
&\leq\norm{Q^\alpha}_{\mathcal{L}(\X)}\norm{Q^{-\alpha}F(x)-Q^{-\alpha}F(y)}\leq \norm{Q^\alpha}_{\mathcal{L}(\X)}K_{F,\alpha}\norm{x-y},
\end{align*}
where $K_{F,\alpha}$ is the Lipschitz constant of $Q^{-\alpha}F$. We set $L_{F,\alpha}:=\|Q^\alpha\|_{\mathcal{L}(\X)}K_{F,\alpha}$.  We can, and do, assume $Q^{-\alpha}F\in C^1_b(\X;\X)$, the general case follows by standard approximation arguments as in the proof of Theorem \ref{Hstrong}.

We will show some preliminary results which will be useful. By Theorem \ref{smild}, the stochastic partial differential equation \eqref{eqF} has a unique mild solution $X(t,x)$. If $F\in C^1_b(\X;\X)$, by Theorem \ref{desmild}, the map $x\mapsto X(\cdot,x)$ from $\X$ to $\X^2([0,T])$ is Gateaux differentiable along any $k\in\X$ for every $p\geq 2$, and the process $Y(t,k)=\J^G X(t,x)k$ is the unique mild solution of 
\begin{gather}\label{cip-cop}
\eqsys{ dY(t,k)=(AY(t,k)+\J F(X(t,x))Y(t,k))dt, & t\in(0,T];\\
 Y(0,k)=k\in\X.}
\end{gather}
In the same way as in the proof of Theorem \ref{deZ}, using the contraction mapping theorem in the space $\X^2([0,T])$ and the Gronwall inequality, we obtain that $\J^G X(t,x)k$ belongs to $\X^2([0,T])$ and for every $t\in [0,T]$ and $x,k\in\X$,
\begin{equation}\label{stigroX}
\|\J^G X(t,x)k\|\leq e^{\int_0^t\|\D F( X(t,x))ds\|_{\mathcal{L}(\X)}}\norm{k}\leq e^{L_{F,\alpha}T}\norm{k}.
\end{equation} 
Now let us prove some results that will be useful in case $\alpha\in [0,1/2)$.
\begin{lemm}\label{abc}
Assume Hypotheses \ref{hyp0} hold true and let $F:\X\ra\X$ be such that $F(\X)\subseteq H_\alpha$, $Q^{-\alpha}F$ is Lipschitz continuous and $F\in C^1_b(\X;\X)$. If $X(t,x)$ is the mild solution of \eqref{eqF}, then the following hold true:
\begin{enumerate}[\rm (a)]
\item \label{a} if $\alpha\in [0,1/2)$, then $\J^G X(s,x)k$ belongs to $H_\alpha$;

\item for $\alpha\in[0,1/2)$, $k\in\X$ and $t\in[0,T]$ it holds
\begin{align}\label{c}
\int_0^t\Vert Q^{-\alpha}e^{(t-s)A}\J F(X(r,x)) \J^G X(r,x)k\Vert^2dr \leq  TL_{F,\alpha}^2 e^{2L_{F,\alpha}T} \|k\|^2.
\end{align}
\end{enumerate}
\end{lemm}

\begin{proof}
We remark that by Theorem \ref{desmild}, the process $\{\J^G X(t,x)k\}_{t\in[0,T]}$ is well defined and it is the mild solution of \eqref{cip-cop}. This means that 
\begin{align}\label{penna}
\J^G X(t,x)k=e^{tA}k+\int^t_0e^{(t-s)A}\J F(X(s,x))(\J^G X(t,x)k)ds.
\end{align}
We start by proving \eqref{a}. By \cite[Proposition 2.1.1]{LUN1}, if $\alpha\in [0,1/2)$ then for any $t>0$ and $\beta\geq 0$
\[e^{tA}(\X)\subseteq Q^{\beta}(\X).\]
By \eqref{penna}, $\J^G X(s,x)k$ belongs to $H_\alpha$, for every $s\in (0,T]$ and $x,k\in\X$. Observe that \eqref{c} is a consequence of the Lipschitzianity of $Q^{-\alpha}F$ and \eqref{stigroX}.
\end{proof}
The main idea behind the proof of Theorem \ref{Hstrongvar} is to obtain an estimate for $\|\D P_t\varphi(x)\|$, for $\alpha\in [0,1/2)$, and for $\|\D_{1/2} P_t\varphi(x)\|$, for $\alpha=1/2$, independent of $x$. When such an estimate is found, then we proceed as in the proof of Theorem \ref{Hstrong}. When $\alpha\in [0,1/2)$ there is a difference between the case $\alpha\in [0,1/4)$ and the case $\alpha\in [1/4,1/2)$. In the first case we can use \eqref{raaaa} which allows us to get a sharper gradient estimate. Instead, for the second case, we are forced to use other results. We now split the proof of Theorem \ref{Hstrongvar} in three cases.

For $\alpha\in [0,1/4)$, as we have mentioned in Section \ref{introduction}, we present a simpler proof than the one in \cite{BON-FUR1,FUR1}, that exploits the identity $A=-(1/2)Q^{2\alpha-1}$ and the analyticity of the semigroup $e^{tA}$.

\begin{proof}[of Theorem \ref{Hstrongvar} for $\alpha\in [0,1/4)$]
First of all we prove a preliminar result. Recalling that $Q^{-\alpha}=A^{\alpha/(1-2\alpha)}$, then by \cite[Formula (2.1.2)]{LUN1} for $\alpha\in[0,1/4)$, there exists $C_\alpha>0$ such that for any $k\in\X$ and $t\in[0,T]$ it holds
\begin{align}\label{b}
\int_0^t\|Q^{-\alpha}e^{sA}k\|^2ds\leq C_\alpha T^{(1-4\alpha)/(1-2\alpha)}\frac{2\alpha-1}{4\alpha-1}\|k\|^2.
\end{align}
Now we proceed in the same way as in the proof of Proposition \ref{BEL}. By Lemma \ref{abc}\eqref{a}, 
\begin{equation}\label{ciaciacia}
\int_0^t\langle Q^{-\alpha}\J^G X(s,x)k, dW(s)\rangle
\end{equation}
is well defined. By \eqref{b}-\eqref{c} we have
\[ 
\int^t_0 \E \big[\Vert Q^{-\alpha}\J^G X(s,x)k\Vert^2\big]ds<+\infty,
\]
and so, by \cite[Remark 2]{EL-LI1}, \eqref{ciaciacia} is a martingale. Multiplying both sides of  \eqref{raaaa} by \eqref{ciaciacia}, and using the same arguments used in the proof of Theorem \ref{Hstrong} we obtain, for any $k\in H_\alpha$,
\[
\langle \D P(t)\varphi(x),k\rangle=\frac{1}{t}\mathbb{E}\left[\varphi(X(t,x))\int_0^t\langle Q^{-\alpha}\J^G X(s,x)k, dW(s)\rangle\right],
\]
and by the It\^o isometry
\begin{align*}
|\langle \D P(t)\varphi(x), k\rangle|^2&\leq\frac{1}{t^2}\Vert \varphi\Vert^2_\infty\mathbb{E}\left[\int_0^t\Vert Q^{-\alpha}\J^G X(s,x)k\Vert^2 ds\right],
\end{align*}
and so by \eqref{b}-\eqref{c}
\begin{align}\label{kuakua}
\Vert \D P(t)\varphi(x)\Vert\leq  \frac{C}{\sqrt t}\Vert \varphi\Vert_\infty,
\end{align}
where $C^2=2\max \{ C_\alpha T^{(1-4\alpha)/(1-2\alpha)}\frac{2\alpha-1}{4\alpha-1}, TL_{F,\alpha}^2e^{2L_{F,\alpha}T} \}$. With the aid of \eqref{kuakua} we conclude using arguments similar to those of the proof of Theorem \ref{Hstrong}.
\end{proof}

If $\alpha\in [1/4,1/2)$, then \eqref{b} is not verified, so we have to obtain an analogous of \eqref{kuakua} in another way. However, we cannot get this estimate from \eqref{raaaa}. So we need the same results used in \cite{BON-FUR1,FUR1}. We will give just give a hint of the proof.


\begin{proof}[of Theorem \ref{Hstrongvar} for $\alpha\in [1/4,1/2)$]
In view of Hypotheses \ref{hyp0} and Lemma \ref{abc}\eqref{a}, by \cite[Proposition 6]{BON-FUR1} and the chain rule, we have
\begin{align}
\scal{\D P(t)\varphi}{k}&=\E\sq{\scal{\D\varphi(X(t,x))}{\J^G X(t,x)k}}\notag\\
&=\E\sq{\varphi(X(t,x))\scal{\int^t_0e^{(t-s)A}Q
^\alpha dW(s)}{Q_t^{-2\alpha}\J^G X(t,x)k}}\notag\\
&\phantom{aa} -\E\sq{\varphi(X(t,x))\int^t_0 \scal{\J F(X(s,x))e^{(t-s)A}\J^G X(t,x)k}{dW(s)} },\label{varbonbon}
\end{align} 
where $Q_t=Q(\Id-e^{2tA})$. We remark that $e^{tA}(\X)\subseteq Q_t^{1/2}(\X)\subseteq Q_t^\alpha(\X)$ for every $t\in(0,T]$. Recalling that by Hypothesis \ref{hyp0}\eqref{hyp0.3} $\int^t_0\Tr[e^{2(t-s)A}Q^{2\alpha}]ds<+\infty$, by \eqref{stigroX}, Lemma \ref{abc}\eqref{a} and \eqref{varbonbon} we get that there exists $C(t,F)>0$ such that for every $x\in\X$
\begin{align}\label{kuku}
\Vert \D P(t)\varphi(x)\Vert\leq  C(t,F)\Vert \varphi\Vert_\infty.
\end{align}
Using \eqref{kuku} we conclude the proof in a same way as in the case $\alpha\in [0,1/4)$.
\end{proof}

\begin{rmk}
We note that in the case $\alpha\in[1/4,1/2)$ it is not possible to obtain an explicit estimate as the one in \eqref{HlipC}. Indeed even in \cite[Theorem 8]{BON-FUR1} the dependence on $t$ of the constant $C(t,F)$ is implicit.
\end{rmk}

We just need to show Theorem \ref{Hstrongvar} in the case $\alpha=1/2$.

\begin{proof}[of Theorem \ref{Hstrongvar} for $\alpha=1/2$]
In the same way as in the proof of Proposition \ref{BEL}, multiplying both sides of  \eqref{raaaa} by 
\[
\int_0^t\langle\J^G X(s,x)h, Q^{1/2}dW(s)\rangle,
\]
we obtain, for any $h\in H_\alpha$,
\[
\langle Q^{1/2}\D P(t)\varphi(x),h\rangle=\frac{1}{t}\mathbb{E}\left[\varphi(X(t,x))\int_0^t\langle Q^{1/2}\J^G X(s,x)h, dW(s)\rangle\right],
\]and 
\begin{align*}
|\langle Q^{1/2}\D P(t)\varphi(x), h\rangle|^2&\leq\frac{1}{t^2}\Vert \varphi\Vert^2_\infty\mathbb{E}\left[\int_0^t\Vert Q^{1/2}\J^G X(s,x)h\Vert^2 ds\right]\notag\\
&\leq \frac{1}{t^2}\Vert \varphi\Vert^2_\infty\|Q^{1/2}\|^2_{\mathcal{L}(\X)}\mathbb{E}\left[\int_0^t\Vert\J^G X(s,x)h\Vert^2 ds\right].
\end{align*}
By Proposition \ref{dalpha} we have 
\begin{equation}\label{tre}
\Vert Q^{1/2}\D P(t)\varphi(x)\Vert=\Vert Q\D P(t)\varphi(x)\Vert_{1/2}=\Vert \D_{1/2} P(t)\varphi(x)\Vert_{1/2}
\end{equation}
and so by \eqref{stigroX}-\eqref{tre} we obtain 
\begin{align}\label{stimagradlip}
\Vert \D_{1/2} P(t)\varphi(x)\Vert_{1/2}\leq  \frac{e^{L_FT}}{\sqrt t}\|Q^{1/2}\|_{\mathcal{L}(\X)}\Vert \varphi\Vert_\infty.
\end{align}
With the aid of \eqref{stimagradlip} we conclude using arguments similar to those of the proof of Theorem \ref{Hstrong}.
\end{proof}

\section{Comparisons with some results in the literature}\label{FUR}

In this section we will relate the results of this paper to those already known in the literature. We stress that the commutation between $A$ and $Q$ helps us to simplify some calculations appearing in this paper.

\subsection{Comparisons with \cite{BON-FUR1,FUR1}}\label{Cfurmi} In \cite{BON-FUR1} and \cite{FUR1} the transition semigroup $P(t)$ of the stochastic partial differential equation
\begin{gather}\label{eqFU}
\eqsys{
dX(t,x)=\big(AX(t,x)+RF(X(t,x))\big)dt+ RdW(t), & t\in(0,T];\\
X(0,x)=x\in \X,
}
\end{gather}
is studied and it is shown that 
\begin{equation}\label{Str}
P(t)\left(B_b(\X)\right)\subseteq \lip_b(\X),\qquad \forall t\in(0,T]
\end{equation} 
under the following hypotheses.

\begin{hyp}\label{hypF}
$ $
\begin{enumerate}[\rm(i)]

\item $A: D(A)\subseteq \X\ra\X$ is the infinitesimal generator of a strongly continuous semigroup $e^{tA}$;

\item $R$ is a linear and continuous operator on $\X$, and
\[
Q_tx:=\int_0^te^{sA}RR^*e^{sA^*}xds,
\]
are trace class operator, for any $t\in [0,T]$;

\item for every $t\in(0,T]$, the semigroup $e^{tA}$ is a Hilbert--Schmidt operator and there exists $k>0$ such that
\begin{align*}
\int_0^ts^{-k}\Tr[e^{sA}RR^*e^{sA^*}]ds<+\infty;
\end{align*}

\item\label{hypFO} for $t\in(0,T]$ the range of $e^{tA}$ is contained in the range of $Q^{1/2}_t$;

\item\label{hypFO1} $F:\X\mapsto\X$ is Fr\'echet differentiable function with bounded gradient.
\end{enumerate}
\end{hyp}
To prove \eqref{Str}, in \cite{BON-FUR1} and \cite{FUR1} (and in many other papers, see for example \cite{CER1,GOZ1,MAS1}) the authors use the Girsanov theorem to make a change of variable in order to exploit the regularity results of the transition semigroup $T(t)$ associated to \eqref{eqFU} with $F=0$. Indeed we recall that, for any $t>0$, we have
\begin{equation}\label{strongOOU}
T(t)\left(B_b(\X)\right)\subseteq \lip_b(\X).
\end{equation} 
Hypothesis \ref{hypF}\eqref{hypFO} is needed to guarantee \eqref{strongOOU} (see, for example \cite[Section 8.3.1]{CER1}, \cite[Section 10.3]{DA-ZA1} and \cite{GOZ1}).
Clearly the hypotheses on $F$ of Theorem \ref{Hstrong} are significantly different from Hypothesis \ref{hypF}\eqref{hypFO1} and consequently also the results on the transition semigroup $P(t)$ that are obtained are different. Instead, for $\alpha\in [0,1/2)$ the hypotheses of Theorem \ref{Hstrongvar} are covered by Hypotheses \ref{hypF}. Indeed it is enough to set
 \[
R=Q^{\alpha},\qquad   A=-\frac{1}{2}Q^{2\alpha-1}.
\]
and to recall that, by \cite[Proposition 2.1.1(i)]{LUN1}, for any $\beta\geq 0$ and $t>0$, we have
\[
e^{tA}(\X)\subseteq Q^{\beta(1-2\alpha)}(\X).
\]
Of course, in this paper the relation $A=-(1/2)Q^{2\alpha-1}$ simplifies the calculus. However, our approach is different since we do not use the Girsanov theorem.
Finally, as we just said in Section \ref{introduction}, the case $\alpha=1/2$ is not covered by the Hypotheses \ref{hypF}. In particular Hypothesis \ref{hypF}\eqref{hypFO} is not verified, since $A=-(1/2)\Id$. This lack of regularity is not somethig related to the function $F$. Indeed, it is known that the transition semigroup $M(t)$ associated to 
\begin{gather*}
\eqsys{
dY(t,x)=-\frac{1}{2}Y(t,x)dt+ Q^{1/2}dW(t), & t\in(0,T];\\
Y(0,x)=x\in \X,
}
\end{gather*}
is an Ornstein--Uhlenbeck semigroup defined by the Mehler formula
\[
M(t)f(x):=\int_X f(e^{-t}x+\sqrt{1-e^{-2t}}y)\gamma(dy),\qquad x\in\X,\ f\in B_b(\X);
\]
where $\gamma$ is the Gaussian measure on $\X$ with mean zero and covariance operator $Q$ and it regularizes only along $Q^{1/2}(\X)$ (see for example \cite[Proposition 2.3]{CER-LUN1}). Hence with $\alpha=1/2$ we cannot hope to achive a result similar to \eqref{Str}.

\subsection{Comparisons with \cite{FE-GO2,FE-GO1,MAS2,MAS1}}\label{Cmasi} 
In \cite{FE-GO2,FE-GO1,MAS2,MAS1} the authors work in a very general setting: $Q$ and $A$ are not linked by any relationship and $\X$ is a separable Banach space with a Schauder basis. They define the following differential operator.
\begin{defi}\label{G-der}
 Let $f:\X\ra\R$ be a continuous function, the $Q^\alpha$-directional derivative $\nabla^{Q^\alpha}f(x;y)$ at a point $x\in\X$ in the direction $y\in\X$ is defined as:
\begin{align*}
\nabla^{Q^\alpha}f(x;y):=\lim_{s\ra 0}\frac{f(x+sQ^\alpha y)-f(x)}{s},
\end{align*}
provided that the limit exists and the map $y\mapsto \nabla^{Q^\alpha}f(x;y)$ belongs to $\X^*$.
\end{defi}
\noindent Furthermore the authors of \cite{FE-GO2,FE-GO1,MAS2,MAS1} assume the following.

\begin{hyp}\label{hyp_mas}
Let $f:\X\ra \X$ be a continuous function  such that $A+f-\eta$ is dissipative on $\X$ for some $\eta\in\R$ and there exists $k\geq 0$ such that $\|f(x)\|\leq c(1+\|x\|^k)$ for some positive constant $c$. Moreover assume that $f(x)\in Q^\alpha(\X)$ and let $F(x)=Q^{-\alpha}f(x)$. Finally assume that $F:\X\ra\X$ is a continuous and Gateaux differentiable function with continuous directional derivatives, and there exists $j\geq0$ such that for every $x,y\in\X$
\begin{align*}
\|F(x)\|\leq c(1+\|x\|^j),\qquad \|(\J^G F(x))y\|\leq c(1+\|x\|^j)\|y\|;
\end{align*}
for some positive constant $c$.
\end{hyp}

Using Hypotheses \ref{hyp_mas} (and hypotheses on $Q$ similar to Hypotheses \ref{hypF}) the authors of \cite{FE-GO2,FE-GO1,MAS2,MAS1} prove that, for every $\varphi\in B_b(\X)$, the function $P(t)\varphi$ admits $Q^\alpha$-directional derivatives in every direction $y\in\X$.  We stress that if $f$ is differentiable along $H_\alpha$, then its $Q^\alpha$-directional derivatives exists and
\[\nabla^{Q^\alpha}f(x;y)=\langle Q^{-\alpha}\D_\alpha f(x),y\rangle.\]
Instead if $f$ admits $Q^\alpha$-directional derivatives, it may be not differentiable along $H_\alpha$. We remark that the derivatives operator defined in Definition \ref{derivataaaaa} is a sort of Fr\'echet derivative along $H_\alpha$, while \eqref{G-der} are Gateaux derivatives along the direction of $H_\alpha$. Finally we stress that in this paper we obtain a Lipschitzianity result (see Theorem \ref{Hstrong} and Theorem \ref{Hstrongvar}), instead, in \cite{FE-GO2,FE-GO1,MAS2,MAS1}, the authors cannot achive a similar result, with Definition \ref{G-der}.

\section{Examples}\label{exA}

In this section we will give some examples to which the results of this paper can be applied.

\subsection{An example for $\alpha\in [0,1/2)$}\label{projection_ex} Let $\alpha\in [0,1/2)$ and let $\Pi\in\mathcal{L}(H_\alpha)$. Let $\beta\geq \alpha$ and consider the map $F:\X\ra\X$ defined as
\begin{align}\label{proj}
F(x):=\eqsys{
\Pi(Q^\beta x), & x\in H_\alpha;\\
0, & x\in \X\setminus H_\alpha.
}
\end{align}
We claim that $F$ satisfies Hypoteses \ref{hyp1}. Indeed, since $F_{|_{H_\alpha}}$ is continuous, then recalling Proposition \ref{propHalpha}\eqref{Halpha6} we obtain that $F$ is Borel measurable. If $x,h\in H_\alpha$ then
\begin{align*}
\|F(x+h)-F(x)\|_\alpha &=\|\Pi(Q^\beta x+Q^\beta h)-\Pi(Q^\beta x)\|_\alpha\leq  \|\Pi\|_{\mathcal{L}(H_\alpha)} \|Q^\beta\|_{\mathcal{L}(H_\alpha)}\norm{h}_\alpha.
\end{align*}
While if $x\in\X\setminus H_\alpha$ and $h\in H_\alpha$, recalling that $x+h\in\X\setminus H_\alpha$, we get 
\begin{align*}
\|F(x+h)-F(x)\|_\alpha &=0.
\end{align*}
So $F$ is $H_\alpha$-Lipschitz. Now if $t\in(0,T]$, $x\in\X$ and $h\in H_\alpha$ since $\alpha\in [0,1/2)$ we know that $e^{tA}x$ belongs to $H_\alpha$ and by Proposition \ref{propHalpha}\eqref{Halpha4} so
\begin{align*}
\|F(e^{tA}x+h)\|^2_\alpha &=\|\Pi(Q^\beta (e^{tA}x)+Q^{\beta}h)\|^2_\alpha\\
&=\|\Pi(e^{tA}(Q^\beta x)+Q^{\beta}h)\|^2_\alpha\\
&\leq  \|\Pi\|_{\mathcal{L}(H_\alpha)} \|e^{tA}(Q^\beta x)+Q^{\beta}h\|^2_\alpha\\
&\leq 2 \|\Pi\|_{\mathcal{L}(H_\alpha)}\|e^{tA}(Q^\beta x)\|^2_\alpha+2 \|\Pi\|_{\mathcal{L}(H_\alpha)}\|Q^{\beta}h\|^2_\alpha\\
&\leq 2 \|\Pi\|_{\mathcal{L}(H_\alpha)}\|Q^\beta x\|^2_\alpha+2 \|\Pi\|_{\mathcal{L}(H_\alpha)}\|Q^{\beta}h\|^2_\alpha\\
&\leq 2 \|\Pi\|_{\mathcal{L}(H_\alpha)}\|Q^\beta x\|^2_\alpha+2 \|\Pi\|_{\mathcal{L}(H_\alpha)}\|Q^\beta\|_{\mathcal{L}(H_\alpha)}\|h\|^2_\alpha\\
&\leq 2 \|\Pi\|_{\mathcal{L}(H_\alpha)}\max\set{\|Q^\beta x\|^2_\alpha,\|Q^\beta\|_{\mathcal{L}(H_\alpha)}}(1+\|h\|^2_\alpha).
\end{align*}
This concludes the proof of our claim. So if we assume Hypotheses \ref{hyp0} then for every $x\in\X$, by Theorem \ref{Hsmild}, the stochastic partial differential equation 
\begin{gather*}
\eqsys{
dX(t,x)=\big(AX(t,x)+F(X(t,x))\big)dt+ Q^{\alpha}dW(t), & t\in(0,T];\\
X(0,x)=x\in \X,
}
\end{gather*}
has a unique mild solution $X(t,x)$ in $\X^2([0,T])$. In particular, applying Theorem \ref{Hstrong}, the transition semigroup $P(t)\varphi(x)=\E[\varphi(X(t,x))]$, defined for $\varphi\in B_b(\X)$, maps the space $B_b(\X)$ in the space $\lip_{b,H_\alpha}(\X)$ for every $t\in(0,T]$. We stress that the function $F$ defined in \eqref{proj}, is not continuous on $\X$ so the classical theory of stochastic partial differential equation cannot be used.

\subsection{An example for $\alpha=1/2$}\label{ex_1/2} Consider the space $\X=L^2([0,1],d\xi)$ where $d\xi$ denotes the Lebesgue measure on $[0,1]$ and let $Q:L^2([0,1],d\xi)\ra L^2([0,1],d\xi)$ be the positive and self-adjoint operator defined as
\[Q f(\xi)=\int_0^1 \max\set{\xi,\eta}f(\eta)d\eta.\]
We emphasize that we have assumed as $Q$ the covariance operator of the Wiener measure on $L^2([0,1],d\xi)$, but we could consider any $Q$ such that $Q(L^2([0,1],d\xi))\subseteq W^{1,2}_0([0,1],d\xi)$, where $W^{1,2}_0([0,1],d\xi)$ is the set of the real-valued functions $f$ defined on $[0,1]$ such that $f$ is absolutely continuous, $f'\in L^2([0,1],d\xi)$ and $f(0)=0$.
If $\alpha=1/2$, it is known that Hypotheses \ref{hyp0} hold true and $H_{1/2}$ is the space $W^{1,2}_0([0,1],d\xi)$. Moreover the norm $\norm{\cdot}_{1/2}$ is equivalent to the norm
\[\norm{f}_{W^{1,2}_0([0,1],d\xi)}:=\|f'\|_{L^2([0,1],d\xi)}.\]
For all these results see \cite[Remark 2.3.13 and Lemma 2.3.14]{BOGIE1}. 

Let $T>0$ and let $g:[0,1]\ra [0,1]$ be a non-decreasing and  Lipschitz continuous function with Lipschitz constant $L_g$. Consider the stochastic partial differential equation
\begin{gather}\label{ex2}
\eqsys{
dX(t,f)=\pa{-\frac{1}{2}X(t,f)+F(X(t,f))}dt+ Q^{1/{2}}dW(t), & t\in(0,T];\\
X(0,f)=f\in L^2([0,1],d\xi),
}
\end{gather}
where $F:L^2([0,1],d\xi)\ra L^2([0,1],d\xi)$ is defined as
\begin{align*}
F(f):=\eqsys{
f\circ g-f(g(0)), & f\in W^{1,2}_0([0,1],d\xi);\\
0, & \text{otherwise}.
}
\end{align*}
We claim that $F$ satisfies Hypotheses \ref{hyp1}. Indeed by \cite[Proposition 129]{PAP1} the function $f\circ g-f(g(0))$ is absolutely continuous and it maps zero to itself. Moreover
\begin{align*}
\|f\circ g-f(g(0))\|^2_{W^{1,2}_0([0,1],d\xi)}&=\int_0^1|(f\circ g)'(\xi)|^2 d\xi\\
&=\int_0^1|f'(g(\xi))g'(\xi)|^2 d\xi\\
&\leq L_g\int_0^1|f'(g(\xi))|^2 |g'(\xi)| d\xi\\
&=L_g\int_0^1|f'(\eta)|^2 d\eta<+\infty.
\end{align*} 
So $F(L^2([0,1],d\xi))$ is contained in $W_0^{1,2}([0,1],d\xi)$. If $f,h\in W^{1,2}_0([0,1],d\xi)$ we get
\begin{align*}
\|F(f+h)-F(f)\|^2_{W^{1,2}_0([0,1],d\xi)} &=\int_0^1|((f+h)\circ g)'(\xi)-(f\circ g)'(\xi)|^2d\xi\\
&=\int_0^1|f'(g(\xi))g'(\xi)+h'(g(\xi))g'(\xi)-f'(g(\xi))g'(\xi)|^2 d\xi\\
&=\int_0^1|h'(g(\xi))g'(\xi)|^2 d\xi\\
&\leq L_g\int_0^1|h'(g(\xi))|^2|g'(\xi)| d\xi\\
&=L_g\int_0^1|h'(\eta)|^2 d\eta=L_g\|h\|^2_{W^{1,2}_0([0,1],d\xi)}.
\end{align*}
While if $f\in L^2([0,1],d\xi)\setminus W^{1,2}_0([0,1],d\xi)$ and $h\in W^{1,2}_0([0,1],d\xi)$ we get 
\[\|F(f+h)-F(f)\|^2_{W^{1,2}_0([0,1],d\xi)}=0.\] 
Then, $F$ is $W^{1,2}_0([0,1],d\xi)$-Lipschitz. Now let $t\in [0,T]$, $f\in L^2([0,1],d\xi)$ and $h\in W^{1,2}_0([0,1],d\xi)$. If $f$ belongs to $L^2([0,1],d\xi)\setminus W^{1,2}_0([0,1],d\xi)$ then 
\begin{align*}
\|F(e^{-t/2}f+h)\|^2_{W^{1,2}_0([0,1],d\xi)}=0,
\end{align*}
while if $f\in W^{1,2}_0([0,1],d\xi)$ then
\begin{align*}
\|F(e^{-t/2}f+h)\|^2_{W^{1,2}_0([0,1],d\xi)} &=\int_0^1|((e^{-t/2}f+h)\circ g)'(\xi)|^2d\xi\\
&=\int_0^1|(e^{-t/2}f'(g(\xi))+h'(g(\xi)))g'(\xi)|^2d\xi\\
&\leq L_g\int_0^1|(e^{-t/2}f'(g(\xi))+h'(g(\xi)))|^2|g'(\xi)| d\xi\\
&=L_g\int_0^1|e^{-t/2}f'(\eta)+h'(\eta)|^2d\eta\\
&\leq 2L_g\int_0^1|e^{-t/2}f'(\eta)|^2d\eta+2L_g\int_0^1|h'(\eta)|^2d\eta\\
&\leq 2L_g\|f\|^2_{W^{1,2}_0([0,1],d\xi)}+2L_g\|h\|^2_{W^{1,2}_0([0,1],d\xi)}\\
&\leq 2L_g\max\set{1,\|f\|^2_{W^{1,2}_0([0,1],d\xi)}}\pa{1+\|h\|^2_{W^{1,2}_0([0,1],d\xi)}}
\end{align*}
Finally using the same arguments as in Section \ref{projection_ex} we obtain that $F$ is Borel measurable. So $F$ satisfies Hypotheses \ref{hyp1}. 

For every  $x\in\X$ and $T>0$, by Theorem \ref{Hsmild}, the stochastic partial differential equation \eqref{ex2} has a unique mild solution $X(t,x)$ in $\X^2([0,T])$. In particular, applying Theorem \ref{Hstrong}, the transition semigroup $P(t)\varphi(x)=\E[\varphi(X(t,x))]$, defined for $\varphi\in B_b(L^2([0,1],d\xi))$, maps the space $B_b(L^2([0,1],d\xi))$ in the space $\lip_{b,{W^{1,2}_0([0,1],d\xi)}}(L^2([0,1],d\xi))$ for every $t\in(0,T]$. We stress that the function $F$ defined in \eqref{proj} is not continuous on $\X$ so the classical theory of stochastic partial differential equations cannot be used. Furthermore the results of \cite{BON-FUR1} and \cite{FUR1} cannot be used since $\alpha=1/2$ as we already remarked in Section \ref{FUR}.

\subsection{An example for Remark \ref{x_dip}} We consider the same setting of Section \ref{ex_1/2}. Let $Y$ be the set of absolutely continuous functions $f:[0,1]\ra\R$ such that $f'$ is bounded and $f(0)=0$. Let $T>0$ and let $g:[0,1]\ra \R$ be a Lipschitz continuous function with Lipschitz continuous derivative. We denote by $L_g$ and $L_{g'}$ the Lipschitz constants of $g$ and $g'$, respectively. Consider the stochastic partial differential equation
\begin{gather}\label{ex3}
\eqsys{
dX(t,f)=\pa{-\frac{1}{2}X(t,f)+F(X(t,f))}dt+ Q^{1/{2}}dW(t), & t\in(0,T];\\
X(0,f)=f\in L^2([0,1],d\xi),
}
\end{gather}
where $F:L^2([0,1],d\xi)\ra L^2([0,1],d\xi)$ is defined as
\begin{align*}
F(f):=\eqsys{
g\circ f-g(f(0)), & f\in Y;\\
0, & \text{otherwise}.
}
\end{align*}
We claim that $F$ satisfies the conditions of Remark \ref{x_dip}. Indeed by \cite[Exercise 17 of Section 5.4]{ROY1} the function $g\circ f-g(f(0))$ is absolutely continuous and it maps zero to itself. Moreover
\begin{align*}
\|g\circ f-g(f(0))\|^2_{W^{1,2}_0([0,1],d\xi)}&=\int_0^1|(g\circ f)'(\xi)|^2 d\xi\\
&=\int_0^1|g'(f(\xi))f'(\xi)|^2 d\xi\\
&\leq L_g^2\int_0^1|f'(\eta)|^2 d\eta=L_g^2\|f\|^2_{W^{1,2}_0([0,1],d\xi)}.
\end{align*} 
So $F(L^2([0,1],d\xi))$ is contained in $W_0^{1,2}([0,1],d\xi)$. If $f\in Y$ and $h\in W^{1,2}_0([0,1],d\xi)$ we get
\begin{align*}
\|F(f+h)-F(f)\|^2_{W^{1,2}_0([0,1],d\xi)} &=\int_0^1|(g\circ (f+h))'(\xi)-(g\circ f)'(\xi)|^2d\xi\\
&=\int_0^1|g'(f(\xi)+h(\xi))(f'(\xi)+h'(\xi))-g'(f(\xi))f'(\xi)|^2 d\xi\\
&\leq 2\int_0^1|g'(f(\xi)+h(\xi))(f'(\xi)+h'(\xi))-g'(f(\xi)+h(\xi))f'(\xi)|^2 d\xi\\
&\phantom{aaaaaaaaaaaaa1}+2\int_0^1|(g'(f(\xi)+h(\xi))-g'(f(\xi)))f'(\xi)|^2 d\xi\\
&\leq 2L_g^2\int_0^1|h'(\xi)|^2d\xi+2L_{g'}^2\int_0^1|h(\xi)f'(\xi)|^2d\xi\\
&\leq 2L_g^2\|h\|^2_{W^{1,2}_0([0,1],d\xi)}+2L_{g'}^2\|f'\|_\infty^2\int_0^1|h(\xi)|^2d\xi\\
&\leq 2L_g^2\|h\|^2_{W^{1,2}_0([0,1],d\xi)}+2L_{g'}^2\|f'\|_\infty^2\int_0^1\abs{\int_0^\xi h'(\eta)d\eta}^2d\xi\\
&\leq 2L_g^2\|h\|^2_{W^{1,2}_0([0,1],d\xi)}+2L_{g'}^2\|f'\|_\infty^2\int_0^1\xi\int_0^\xi |h'(\eta)|^2d\eta d\xi\\
&\leq 2L_g^2\|h\|^2_{W^{1,2}_0([0,1],d\xi)}+2L_{g'}^2\|f'\|_\infty^2\|h\|^2_{W^{1,2}_0([0,1],d\xi)}\\
&\leq 2\max\set{L_g^2,L_{g'}^2\|f'\|_\infty^2}\|h\|^2_{W^{1,2}_0([0,1],d\xi)}
\end{align*}
While if $f\in L^2([0,1],d\xi)\setminus Y$ and $h\in W^{1,2}_0([0,1],d\xi)$ we get 
\[\|F(f+h)-F(f)\|^2_{W^{1,2}_0([0,1],d\xi)}=0.\] 
These imply that $F$ is $W^{1,2}_0([0,1],d\xi)$-Lipschitz. Now let $t\in[0,T]$, $f\in L^2([0,1],d\xi)$ and $h\in W^{1,2}_0([0,1],d\xi)$. If $f$ belongs to $L^2([0,1],d\xi)\setminus W^{1,2}_0([0,1],d\xi)$ then 
\begin{align*}
\|F(e^{-t/2}f+h)\|^2_{W^{1,2}_0([0,1],d\xi)}=0.
\end{align*}
While if $f\in W^{1,2}_0([0,1],d\xi)$ then
\begin{align*}
\|F(e^{-t/2}f+h)\|^2_{W^{1,2}_0([0,1],d\xi)} &=\int_0^1|( g\circ (e^{-t/2}f+h))'(\xi)|^2d\xi\\
&=\int_0^1|g'(e^{-t/2}f(\xi)+h(\xi))(e^{-t/2}f'(\xi)+h'(\xi))|^2d\xi\\
&\leq L_g^2\int_0^1|e^{-t/2}f'(\eta)+h'(\eta)|^2d\eta\\
&\leq 2L_g^2\int_0^1|e^{-t/2}f'(\eta)|^2d\eta+2L_g^2\int_0^1|h'(\eta)|^2d\eta\\
&\leq 2L_g^2\|f\|^2_{W^{1,2}_0([0,1],d\xi)}+2L_g^2\|h\|^2_{W^{1,2}_0([0,1],d\xi)}\\
&\leq 2L_g^2\max\set{1,\|f\|^2_{W^{1,2}_0([0,1],d\xi)}}\pa{1+\|h\|^2_{W^{1,2}_0([0,1],d\xi)}}
\end{align*}
Finally using the same arguments as in Section \ref{projection_ex} we obtain that $F$ is Borel measurable. So $F$ satisfies the conditions of Remark \ref{x_dip}. 

For every $x\in\X$ and $T>0$, by Theorem \ref{Hsmild}, the stochastic partial differential equation \eqref{ex3} has a unique mild solution $X(t,x)$ in $\X^2([0,T])$. In particular, by Remark \ref{x_dip}, the transition semigroup $P(t)\varphi(f)=\E[\varphi(X(t,f))]$, defined for $\varphi\in B_b(L^2([0,1],d\xi))$, satisfies
\[\abs{P(t)\varphi(f+h)-P(t)\varphi(f)}\leq \frac{e^{\sqrt{2}\max\set{L_g,L_{g'}\|f'\|_\infty} T}}{\sqrt{t}}\norm{\varphi}_{\infty}\norm{h}_\alpha,\]
whenever $t\in(0,T]$, $f\in Y$ and $h\in W^{1,2}_0([0,1]d\xi)$, while if $f\in L^2([0,1],d\xi)\setminus Y$, $t\in(0,T]$ and $h\in W^{1,2}_0([0,1]d\xi)$, then 
\[\abs{P(t)\varphi(f+h)-P(t)\varphi(f)}\leq \frac{1}{\sqrt{t}}\norm{\varphi}_{\infty}\norm{h}_\alpha.\]

\subsection{A gradient type perturbation} Assume that Hypotheses \ref{hyp0} hold true and consider the function $F:\X\ra\X$ defined by
\[F(x)=Q^\alpha\D U(x)\]
for some convex, Fr\'echet differentiable with Lipschitz continuous Fr\'echet derivative function $U:\X\ra\R$. This type of function $F$ is pretty common in the literature (see for example \cite{AD-CA-FE1,AN-FE-PA1,CAP-FER1,CAP-FER2,DA2,DA-LU3,DA-TU1,DA-ZA1,FER1}). It is easy to see that the hypotheses of Theorem \ref{Hstrongvar} are satisfied. Indeed, it is obvious that $F(\X)\subseteq H_\alpha$. Moreover 
\begin{align*}
\|F(x)-F(y)\|&=\|Q^\alpha\D U(x)-Q^\alpha\D U(y)\|\\
&\leq \|Q^\alpha\|_{\mathcal{L}(X)}\|\D U(x)-\D U(y)\|\\
&\leq \|Q^\alpha\|_{\mathcal{L}(X)} L_{\D U}\|x-y\|,
\end{align*}
where $L_{\D U}$ is the Lipschitz constant of $\D U$. So for every and  $x\in\X$, by Theorem \ref{smild}, the stochastic partial differential equation 
\begin{gather*}
\eqsys{
dX(t,x)=\big(AX(t,x)+F(X(t,x))\big)dt+ Q^{\alpha}dW(t), & t\in(0,T];\\
X(0,x)=x\in \X,
}
\end{gather*}
has a unique mild solution $X(t,x)$ in $\X^2([0,T])$. In particular, applying Theorem \ref{Hstrongvar}, the transition semigroup $P(t)\varphi(x)=\E[\varphi(X(t,x))]$, defined for $\varphi\in B_b(\X)$, maps the space $B_b(\X)$ in the space $\lip_{b,H_\alpha}(\X)$ for every $t\in(0,T]$. We remark that this result, when $\alpha\in[0,1/2)$, was already proved in \cite{FUR1} and \cite{MAS2}, while if $\alpha=1/2$ it is new.

\subsection{Cahn--Hilliard type equations}\label{ciancian}
Cahn-Hilliard stochastic equations such as
\[
du(t,x)=\big(\Delta^2u(t,x)-\Delta f(u(t,x))\big)dt+dW(t),\qquad (t,x)\in\R^+\times[0,\pi]^d
\]
where $d\in\N\setminus\set{0}$, $f$ is a polynomial of odd degree with positive leading coefficient and $u:\R^+\times [0,\pi]^d\ra\R$, where considered in \cite{CA1,DA-DE1}. Here we consider an abstract generalization already studied in \cite{DA-DE-TU1} and \cite{ES-ST1}, 
\begin{gather}\label{CH}
\eqsys{
dX(t,x)=\big(AX(t,x)+(-A)^{1/2}F(X(t,x))\big)dt+ dW(t), & t\in(0,T];\\
X(0,x)=x\in \X.
}
\end{gather}
We assume that Hypotheses \ref{hyp0} hold with $\alpha=0$. Let $F:\X\ra\X$ be such that $F(\X)\subseteq H_{1/2}$, and for every $x,k\in\X$, there exists a constant $L_F>0$ such that
\[ 
\norm{F(x+k)-F(x)}_{1/2}\leq L_F\norm{k}.
\] 
Then, recalling that $A=-(1/2)Q^{-1}$, we apply Theorems \ref{Hstrong} on the transition semigroup $P(t)$ associated  to \eqref{CH}, and so $P(t)$ maps the space $B_b(\X)$ in $\lip_{b}(\X)$ for every $t\in(0,T]$.


\subsection{A classical example} Again, we consider the setting of Section \ref{ex_1/2}. Let $F:L^2([0,1],d\xi)\ra L^2([0,1],d\xi)$ defined by choosing $x_1,\ldots, x_n\in L^2([0,1],d\xi)$ and a function $f:[0,1]\times\R^n\ra\R$, $(\xi,y_1,\ldots,y_n)\mapsto f(\xi,y_1,\ldots,y_n)$ and setting
\begin{align*}
(F(g))(\xi):=f\pa{\xi,\int_0^1 g(\eta)x_1(\eta) d\eta,\ldots, \int_0^1 g(\eta)x_n(\eta) d\eta}.
\end{align*}
Assume that $x_1,\ldots, x_n$ are orthonormal and for every $i=1,\ldots,n $
\begin{align*}
&f,\frac{\partial f}{\partial \xi},\frac{\partial f}{\partial y_i}\text{ are bounded and continuous on }[0,1]\times\R^n;\\
&f(0,y_1,\ldots,y_n)=0, \text{ for every }y_1,\ldots,y_n\in\R.
\end{align*}
$F$ satisfies the hypotheses of Theorem \ref{Hstrongvar}. Indeed $F(L^2([0,1],d\xi))\subseteq W^{1,2}_0([0,1],d\xi)$, since
\[(F(g))(0)=f\pa{0,\int_0^1 g(\eta)x_1(\eta) d\eta,\ldots, \int_0^1 g(\eta)x_n(\eta) d\eta}=0\]
and
\[(F(g))'(\xi)=\frac{\partial f}{\partial \xi}\pa{\xi,\int_0^1 g(\eta)x_1(\eta) d\eta,\ldots, \int_0^1 g(\eta)x_n(\eta) d\eta}\leq \norm{\frac{\partial f}{\partial \xi}}_\infty.\]
Moreover for $g_1,g_2\in L^2([0,1],d\xi)$
\begin{align*}
&\|F(g_1)-F(g_2)\|_{L^2([0,1],d\xi)}^2\\
&=\int_0^1\abs{f\pa{\xi,\int_0^1 g_1x_1 d\eta,\ldots, \int_0^1 g_1x_n d\eta}-f\pa{\xi,\int_0^1 g_2x_1 d\eta,\ldots, \int_0^1 g_2x_n d\eta}}\\
&\leq n^2\sup_{i=1,\ldots,n}\set{\norm{\frac{\partial f}{\partial \xi}}^2_\infty,\norm{\frac{\partial f}{\partial y_i}}^2_\infty}\sum_{i=1}^n\abs{\int_0^1(g_1-g_2)x_id\eta}^2\\
&\leq n^2\sup_{i=1,\ldots,n}\set{\norm{\frac{\partial f}{\partial \xi}}^2_\infty,\norm{\frac{\partial f}{\partial y_i}}^2_\infty}\norm{g_1-g_2}^2_{L^2([0,1],d\xi)}.
\end{align*}

By Theorem \ref{smild} for every $x\in\X$, the stochastic partial differential equation 
\begin{gather*}
\eqsys{
dX(t,x)=\big(-\frac{1}{2}X(t,x)+F(X(t,x))\big)dt+ Q^{1/2}dW(t), & t\in(0,T];\\
X(0,x)=x\in \X,
}
\end{gather*}
has a unique mild solution $X(t,x)$ in $\X^2([0,T])$. By Theorem \ref{Hstrongvar}, the transition semigroup maps the space $B_b(\X)$ in $\lip_{b,H_{1/2}}(\X)$ for every $t\in(0,T]$. So we get an improvement of \cite[Section 4]{FUR1}, since there the case $\alpha=1/2$ was not considered.

\appendix

\section{Proof of Theorem \ref{smild}}\label{proofsmild}

The aim of this section is to look for pathwise uniqueness of for \eqref{lalala}.

\begin{proof}[of Theorem \ref{smild}]
We first prove uniqueness. Let $X_1(t,x)$, $X_2(t,x)$ be two mild solutions of \eqref{lalala}. Recall that by definition a mild solution solves
\[
X(t,x)=e^{tA}x+\int_0^te^{(t-s)A}\Phi(s,X(s,x))ds+\int_0^te^{(t-s)A}Q^{\alpha}dW(s).
\]
Hence, we have
\begin{align*}
\E[\norm{X_1(t,x)-X_2(t,x)}^2]&=\E\sq{\norm{\int_0^te^{(t-s)A}\left(\Phi(s,X_1(s,x))-\Phi(s,X_2(s,x))\right)ds}^2}\\ 
&\leq t\E\sq{\int_0^t\norm{e^{(t-s)A}(\Phi(s,X_1(s,x))-\Phi(s,X_2(s,x)))}^2ds}\\
&\leq t\E\sq{\int_0^t\norm{\Phi(s,X_1(s,x))-\Phi(s,X_2(s,x))}^2ds}\\
&\leq tL_\Phi^2\int_0^t\E\sq{\norm{X_1(v,x)-X_2(v,x)}^2}ds.
\end{align*}
From the Gronwall inequality and the same arguments of the proof of \cite[Theorem 7.5]{DA-ZA4} uniqueness follows.

The proof of existence is based on the contraction mapping theorem. We define the Volterra operator
\[V(Y)(t):=e^{tA}x+\int_0^te^{(t-s)A}\Phi(s,Y(s))ds+\int_0^te^{(t-s)A}Q^{\alpha}dW(s),\]
in the space $\X^2([0,T])$, and first of all we show that $V$ maps $\X^2[0,T]$ into itself. Indeed, for any $Y\in \X^2[0,T]$ we have 
\begin{align}
\norm{V(Y)}_{\X^2[0,T]}^2\leq 3\|e^{\cdot A}x\|_{\X^2[0,T]}^2+&3\norm{\int_0^\cdot e^{(\cdot-s)A}\Phi(s,Y(s))ds}_{\X^2[0,T]}^2\notag\\
&\phantom{aaaaaaaaaa}+3\norm{\int_0^\cdot e^{(\cdot -s)A}Q^{\alpha}dW(s)}_{\X^2[0,T]}^2.\label{cincin}
\end{align}
We recall that $\|e^{tA}x\|_{\X^2[0,T]}^2\leq \norm{x}^2$, for $t>0$.
Let $y\in \X$ be such that $\eqref{inty}$ holds, then
\begin{align*}
\norm{\int_0^\cdot e^{(\cdot-s)A}\Phi(s,Y(s))ds}_{\X^2[0,T]}^2 &=\sup_{t\in[0,T]}\E\left[\norm{\int_0^te^{(t-s)A}\Phi(s,Y(s))ds}^2\right]\\
&\leq T\E\left[\int_0^T\norm{\Phi(s,Y(s))}^2ds\right]\\
&\leq 2T\E\left[\int_0^T\norm{\Phi(s,Y(s))-\Phi(s,y)}^2ds\right]\\
&\phantom{aaaaaaaaaaaaaaaaaaaa}+2T\E\left[\int_0^T\norm{\Phi(s,y)}^2ds\right]\\
&\leq 2TL^2_\phi\E\left[\int_0^T\norm{Y(s)-y}^2ds\right]+2T\E\left[\int_0^T\norm{\Phi(s,y)}^2ds\right]\\
&\leq 4T^2L^2_\phi\norm{Y}^2_{\X[0,T]^2}+4T^2L_{\Phi}^2\norm{y}^2+2T\int_0^T\norm{\Phi(s,y)}^2ds.
\end{align*}
Since $Y\in \X^2[0,T]$ and recalling that \eqref{inty} holds,  
\[\norm{\int_0^\cdot e^{(\cdot -s)A}\Phi(s,Y(s))ds}_{\X^2[0,T]}^2<+\infty.\] 
By \cite[Theorem 4.36]{DA-ZA4} and Hypothesis \ref{hyp0}, the third summand in \eqref{cincin} is finite. In the same way as the proof of uniqueness, we have
\[\norm{V(Y_1)-V(Y_2)}_{\X^2[0,T]}^2\leq T^2L^2_\phi \norm{Y_1-Y_2}_{\X^2[0,T]}^2.\]
So the existence follows by the contraction mapping theorem (using similar arguments as the one used in Theorem \ref{deZ}) and same arguments of proof of \cite[Theorem 7.5]{DA-ZA4}.
\end{proof}

\section*{Acknowledgements} The authors would like to thank A. Lunardi for many useful discussions and comments. The authors are members of GNAMPA (Gruppo Nazionale per l'Analisi Matematica, la Probabilit\`a e le loro Applicazioni) of the Italian Istituto Nazionale di Alta Ma\-te\-ma\-ti\-ca (INdAM).  S. F. have been partially supported by the INdAM-GNAMPA
Project 2019 ``Metodi analitici per lo studio di PDE e problemi collegati in dimensione infinita''. The authors have been also partially supported by the research project PRIN 2015233N5A ``Deterministic and stochastic evolution equations'' of the Italian Ministry of Education, MIUR.



\end{document}